\documentclass[12pt]{article}

\usepackage[T1]{fontenc}
\usepackage[utf8]{inputenc}
\usepackage{lmodern}
\usepackage{amsmath,amssymb,amsthm,mathtools}
\usepackage{microtype}
\usepackage{cite}
\usepackage[hidelinks]{hyperref}

\hypersetup{
  pdftitle={Target Geometry in Prescribed-Value Schwarz Lemmas for Harmonic Maps},
  pdfauthor={Miljan Knezevic and Miodrag Mateljevic},
  pdfkeywords={harmonic mappings, prescribed values, Mobius-weighted Hilbert pairing, round affine sections, Hilbert balls, Cayley-Klein geometry, pointwise distortion}
}

\numberwithin{equation}{section}

\theoremstyle{plain}
\newtheorem{theorem}{Theorem}[section]
\newtheorem{lemma}[theorem]{Lemma}
\newtheorem{proposition}[theorem]{Proposition}
\newtheorem{corollary}[theorem]{Corollary}
\theoremstyle{definition}
\newtheorem{definition}[theorem]{Definition}
\newtheorem{example}[theorem]{Example}
\theoremstyle{remark}
\newtheorem{remark}[theorem]{Remark}

\newcommand{\D}{\mathbb D}
\newcommand{\T}{\mathbb T}
\newcommand{\C}{\mathbb C}
\newcommand{\R}{\mathbb R}
\newcommand{\B}{\mathbb B}
\newcommand{\dm}{\,dm}
\newcommand{\re}{\operatorname{Re}}
\newcommand{\Aut}{\operatorname{Aut}}
\newcommand{\CK}{\operatorname{CK}}

\title{Target Geometry in Prescribed-Value Schwarz Lemmas for Harmonic Maps}
\author{Miljan Kne\v{z}evi\'c\thanks{Corresponding author. E-mail: miljan.knezevic@matf.bg.ac.rs}\\
Miodrag Mateljevi\'c\thanks{E-mail: miodrag@matf.bg.ac.rs}\\[0.5em]
\small Faculty of Mathematics, University of Belgrade\\
\small Studentski trg 16, 11000 Belgrade, Serbia}
\date{}

\begin{document}
\maketitle

\begin{abstract}
We study harmonic maps \(F:\D\to G\) into bounded domains in real Hilbert
spaces, prescribing \(F(0)\) and \(dF_0(\R^2)\) when \(dF_0\) is nonzero and
conformal. We prove a target-independent identity for the
M\"obius-weighted Hilbert pairing. After normalization, the pairing is affine
in \(\|dF_0\|\) with positive slope, yielding an exact equivalence between
the derivative extremal problem and a boundary-pairing problem.

For a round affine section, this gives a necessary and sufficient integral
criterion for the M\"obius parametrization to be extremal. A
supporting-hyperplane condition guarantees the required integral inequality
and characterizes equality. Examples show that roundness alone is
insufficient and that the supporting condition is not necessary. For unit
balls of real Hilbert spaces of dimension at least two, including
infinite-dimensional spaces, we obtain the sharp prescribed-value
Schwarz-Pick estimate, all equality cases, and quantitative \(L^2\) boundary
stability.

We also prove a Cayley-Klein contraction under pointwise distortion. For
\(K\geq1\), let \(M_K\) denote the supremum of \(L_F(0)\) over harmonic maps
\(F:\D\to\B_H\) satisfying \(F(0)=0\) and
\(0<\ell_F(0)\leq L_F(0)\leq K\ell_F(0)\), where \(L_F(0)\) and
\(\ell_F(0)\) are the maximal and minimal stretchings at the origin, respectively.
We determine \(M_K\), identify the unique optimizing parameter, and
characterize all extremals. The function \(K\mapsto M_K\) is strictly increasing, with \(M_1=1\) and
\(M_K\to4/\pi\) as \(K\to\infty\).
\end{abstract}



\section{Introduction}

Derivative estimates for harmonic maps depend essentially on the
prescribed value, because conformal changes of target coordinates need
not preserve harmonicity. When the differential at the reference point
is conformal, its image plane provides a second geometric constraint.
We prescribe both the value and this plane and seek the largest possible
operator norm of the differential.

Let $\D\subset\C$ be the unit disk, let $G$ be a bounded domain in a real
Hilbert space $H$ with $\dim H\geq2$, let $p\in G$, and
fix a real two-dimensional linear subspace $\Lambda\subset H$. We consider
the class
\begin{equation}\label{eq:intro-admissible-class}
\begin{aligned}
 \mathcal F_G(p,\Lambda)=\{F:\D\to G\text{ harmonic}:{}& F(0)=p,\ dF_0\ne0\text{ conformal},\\
 &dF_0(\R^2)=\Lambda\},
\end{aligned}
\end{equation}
and its extremal value
\begin{equation}\label{eq:intro-directional-value}
 M_G(p,\Lambda)=\sup_{F\in\mathcal F_G(p,\Lambda)}\|dF_0\|.
\end{equation}
Here $\|dF_0\|$ is the operator norm of the real differential. Conformality
means that $|dF_0v|=\lambda|v|$ for all $v\in\R^2$, with
$\lambda=\|dF_0\|>0$.
The prescribed plane is the image of this differential at the origin;
the map itself may leave the affine plane $\Pi=p+\Lambda$. If
$G\cap\Pi$ is a Euclidean disk, its M\"obius parametrization is a
candidate for the extremal map. The question is whether it remains extremal
among all harmonic disks in $G$ with the same prescribed value and
differential plane.

The key analytic result is a target-independent identity for the
\emph{M\"obius-weighted Hilbert pairing}. It is proved for
Bochner-integrable functions using only their mean and first cosine
and sine Fourier coefficients; no restriction on their range is
imposed at this stage. After normalization of the conformal differential,
the pairing is an affine function of \(\|dF_0\|\) with positive linear
coefficient. Theorem~\ref{thm:directional-paired-equivalence} is the
resulting equivalence theorem: on the normalized class, the derivative norm
and the boundary pairing induce the same ordering and have the same
maximizing sequences and extremals. The geometry of the target enters only
afterwards, through estimates for the boundary pairing.

To state this relation, suppose that $G\cap\Pi$ is a Euclidean disk of
radius $c>0$. Translating \(G\) and \(p\) by the same vector in \(\Lambda\) places
the center of the section at the point of \(\Pi\) nearest the origin.
This translation leaves the differential unchanged and gives a bijection
between the corresponding classes of harmonic maps. We retain the notation
\(G\) and \(p\) after this normalization. Let \(\pi_\Lambda\) be the
orthogonal projection onto \(\Lambda\) and write
\[
 p=p_\Lambda+p_\perp,
 \qquad p_\Lambda=\pi_\Lambda p,
 \qquad p_\perp\perp\Lambda.
\]
Thus
\[
 G\cap\Pi=\{p_\perp+v:v\in\Lambda,\ |v|<c\}.
\]
We set
\[
 a=|p_\Lambda|,\qquad b=|p_\perp|,\qquad
 s=\frac ac\in[0,1),\qquad R=\sqrt{c^2+b^2}.
\]
The number $c$ is the radius of the affine disk. For a point
$y=p_\perp+cu$ on its boundary circle, with $u\in\Lambda$ and $|u|=1$,
orthogonality gives
\[
 |y-p_\perp|=c,\qquad |y|=R.
\]
Thus $R$ is the common distance of these boundary points from the origin.
It equals the section radius $c$ exactly when $p_\perp=0$.
Neither spherical symmetry of $G$ nor containment in the ball of radius
$R$ is assumed.

We choose an orthonormal basis $e_1,e_2$ of $\Lambda$ with
$p_\Lambda=ae_1$, taking $e_1$ arbitrarily when $a=0$.
In the complex coordinates determined by this basis, the normalized
conformal parametrization is
\[
 T(z)=p_\perp+c\frac{s+z}{1+sz},
\]
where the complex summand is regarded as a vector in $\Lambda$.
Thus $T(0)=p$, and $T$ is a conformal diffeomorphism onto the section.
Its differential at the origin has norm $c(1-s^2)$.
This is the map $E_s^c$ defined in \eqref{eq:weighted-E-def}; the choice
of basis also fixes the orientation of its parametrization.

Let $m$ denote normalized Lebesgue measure on $\T=\partial\D$. Every bounded
Hilbert-valued harmonic map $F$ has a Poisson boundary function $\Phi_F$,
unique up to equality almost everywhere; see
Lemma~\ref{lem:hilbert-boundary}. For the fixed parametrization $T$,
the pairing is
\begin{equation}\label{eq:intro-weighted-pairing}
 I_s^c(\Phi)
 =\int_\T |1+s\zeta|^2
 \langle\Phi(\zeta),T(\zeta)\rangle\,dm(\zeta).
\end{equation}
The weight cancels the M\"obius denominator on the boundary.
Theorem~\ref{thm:weighted-pairing-identity} evaluates the resulting
integral in the real Hilbert inner product using the value and derivative
conditions.

Every $F\in\mathcal F_G(p,\Lambda)$ can be precomposed with a rotation
or reflection of $\D$ so that
\[
 F_x(0)=\lambda e_1,\qquad F_y(0)=\lambda e_2,
 \qquad \lambda=\|dF_0\|>0.
\]
These changes preserve the target, the prescribed value, and the
differential norm. We denote this normalized class by
\(\mathcal F_G^{\mathrm n}(p,\Lambda)\). For every
\(F\in\mathcal F_G^{\mathrm n}(p,\Lambda)\),
Theorem~\ref{thm:directional-paired-equivalence} gives
\begin{equation}\label{eq:intro-equivalence}
 I_s^c(\Phi_F)
 =c\|dF_0\|+2csa+(1+s^2)b^2.
\end{equation}
Since $c>0$, the two functionals have the same extremals and maximizing
sequences on the normalized class.

The M\"obius parametrization $T$ belongs to this class. Its boundary
function has constant norm $R$, and
$I_s^c(\Phi_T)=(1+s^2)R^2$. Applying
\eqref{eq:intro-equivalence} to $F$ and $T$ gives the difference formula
in Corollary~\ref{cor:mobius-extremality}: \begin{equation}\label{eq:intro-exact-mobius}
 I_s^c(\Phi_F)-I_s^c(\Phi_T)
 =c\bigl(\|dF_0\|-\|dT_0\|\bigr).
\end{equation}
It follows that $T$ is extremal if and only if
$I_s^c(\Phi_F)\leq(1+s^2)R^2$ for every normalized admissible map $F$.
This is the integral criterion in Corollary~\ref{cor:mobius-extremality}.
Proposition~\ref{prop:section-support-criterion} gives a geometric
sufficient condition that guarantees this inequality for the entire
admissible class.

The sufficient condition is expressed in terms of supporting hyperplanes
along the boundary circle. If
\begin{equation}\label{eq:intro-support-condition}
 \langle X,y\rangle\leq R^2
 \qquad(X\in\overline G,\quad y\in\partial(G\cap\Pi)),
\end{equation}
where the boundary is taken in $\Pi$, then integration against the positive
weight proves the required inequality. Proposition~\ref{prop:section-support-criterion}
therefore determines the extremal value and characterizes equality through
the corresponding contact sets $C_y$, defined in that proposition. If $C_{T(\zeta)}=\{T(\zeta)\}$ for almost every $\zeta\in\T$, then $T$ is
the unique normalized extremal.

Information about the ambient target beyond the section is essential. Example~\ref{ex:round-section-not-sufficient}
shows that the existence of a round affine section alone does not make its
M\"obius parametrization extremal. On the other hand, the supporting-hyperplane
condition concerns the entire target only along one boundary circle and does
not require global spherical symmetry: Example~\ref{ex:rotational-section-support}
gives a bounded convex non-ball target satisfying it. The condition is not
necessary either. For complex-line sections, the same pairing compares
derivatives of holomorphic disks, and
Theorem~\ref{thm:balanced-complex-lines} proves extremality on every central
complex line of a bounded balanced convex domain, including directions in
which the supporting-hyperplane condition fails. Finally,
Theorem~\ref{thm:circular-section-rigidity} shows that requiring every central real two-plane section to be a centered disk forces the target itself to be a Hilbert ball.

For the unit ball $\B_H=\{x\in H:|x|<1\}$, condition
\eqref{eq:intro-support-condition} is
the real Cauchy-Schwarz inequality.
Proposition~\ref{prop:section-support-criterion} yields
\begin{equation}\label{eq:intro-ball-estimate}
 \|dF_0\|\leq
 \frac{1-|p|^2}{\sqrt{1-|p|^2+|\pi_\Lambda p|^2}}.
\end{equation}
The denominator is the radius $c$ of $(p+\Lambda)\cap\B_H$; here $R=1$
is the radius of the target ball. Equality holds if and only if $F$ is a
conformal or anticonformal diffeomorphism onto this affine disk.
The estimate holds in every real Hilbert space of dimension at least two,
even for maps whose ranges are not contained in any finite-dimensional
subspace; see Example~\ref{ex:L2-ball}. Its finite-dimensional
specialization is stated in Corollary~\ref{cor:finite-dimensional-ball}.
Beyond the sharp inequality, Theorem~\ref{thm:ball-boundary-stability}
gives an identity that yields quantitative $L^2$ control of the boundary
distance from the normalized extremal. In particular, for fixed $p$ and
$\Lambda$, normalized maximizing sequences converge in $L^2$.

The M\"obius-weighted Hilbert identity also applies to nonconformal
differentials. We write $L_F(z)$ and $\ell_F(z)$ for the maximal and
minimal stretchings defined in Definition~\ref{def:pointwise-distortion}.
Under the pointwise condition
$0<\ell_F(z)\leq L_F(z)\leq K\ell_F(z)$, we obtain an infinitesimal
Cayley-Klein estimate without an alignment assumption.
Section~\ref{sec:ball-CK} introduces the dimension-free Klein norm and
proves the corresponding contraction estimate. A bound depending on $K$
and the relative position of $p$ and $\Lambda$ implies a uniform estimate
with factor $K$; if the pointwise condition holds throughout $\D$,
integration gives the corresponding distance inequality.

At the center of the ball we determine the optimal constant
\[
\begin{aligned}
 M_K=\sup\{L_F(0):{}&F:\D\to\B_H\text{ harmonic},\ F(0)=0,\\
 &0<\ell_F(0)\leq L_F(0)\leq K\ell_F(0)\}.
\end{aligned}
\]
The parametrization of the principal derivatives used here is related,
after the change of parameter \(\tau=\tan\theta\), to that in
\cite{Zwonek2022}. Under the additional constraint
\(L_F(0)/\ell_F(0)\leq K\), the optimizing parameter is unique, all
extremals are determined up to orthogonal transformations of $H$ and
rotations or reflections of $\D$, and $M_K$ is strictly increasing. More precisely,
\[
 M_1=1,\qquad
 1<M_K<\min\left\{\frac{2K}{K+1},\frac4\pi\right\}\quad(K>1),
\]
and
\[
 \lim_{K\to+\infty}M_K=\frac4\pi.
\]
Thus $M_K$ increases strictly from the conformal value $1$ at $K=1$ to
the unrestricted constant $4/\pi$ as $K\to+\infty$. The corresponding
normalized boundary functions converge to the boundary function of a
one-dimensional harmonic extremal. The unrestricted estimate itself is
proved directly by scalar projection and the Poisson derivative formula.

The planar background includes linear and variational extremal methods
\cite{DurenSchoberVariational,DurenSchober}, harmonic Schwarz-type estimates
and related extremal inequalities
\cite{Heinz,Colonna,DurenBook,KalajVuorinen,Mateljevic2018,
MateljevicSvetlik,Knezevic2025,KMS2025}, and descriptions of admissible
derivatives at the center \cite{KovalevYang,BrevigOrtegaSeip}. The
finite-dimensional ball theorem in \cite{ForstnericKalaj} identifies
affine disks as the equality cases. Our approach isolates the
M\"obius-weighted pairing as a target-independent Hilbert-space identity
before any estimate involving the target is imposed.

Forstneri\v c and Kalaj \cite{ForstnericKalajExtremal} classify the bounded
convex planar pointed domains for which a conformal parametrization is
extremal; their family includes noncircular domains. Their
higher-dimensional formulation fixes one tangent vector at the base point.
Here we prescribe the entire image plane of the conformal differential and
study how the ambient target geometry affects a round affine section.

Section~\ref{sec:preliminaries} fixes the notation and boundary
representation. Section~\ref{sec:weighted-support} proves the pairing
identity, and Section~\ref{sec:target-geometry} establishes the extremal
equivalence and its geometric applications. Section~\ref{sec:disk-local}
treats the disk-valued problem explicitly. Section~\ref{sec:ball-CK}
treats Hilbert balls and the Cayley-Klein estimates.
Section~\ref{sec:cylinder-aligned} treats cylinders and aligned distortion,
and Section~\ref{sec:center-distortion} solves the center problem.

\section{Preliminaries}\label{sec:preliminaries}

We write
\[
 \D=\{z\in\C: |z|<1\},\qquad \T=\partial\D,
\]
and denote by \(m\) normalized Lebesgue measure on \(\T\). Thus
\[
 \int_\T \psi\dm=\frac1{2\pi}\int_0^{2\pi}\psi(e^{it})\,dt.
\]
For \(\Phi\in L^1(\T)\) we use the Fourier convention
\begin{equation}\label{eq:fourier-convention}
 \widehat\Phi(n)=\int_\T \Phi(e^{it})e^{-int}\dm,
 \qquad n\in\mathbb Z.
\end{equation}

Let \(f=u+iv\) be a \(C^1\) complex-valued map in a plane domain. We use
\[
 f_z=\frac12(f_x-if_y),\qquad
 f_{\overline z}=\frac12(f_x+if_y).
\]
The operator norm and minimal stretching are given by the following
elementary formulas.

\begin{lemma}\label{lem:operator-norm}
For every \(C^1\) complex-valued map \(f\) and every point \(z\) in its domain,
\begin{equation}\label{eq:norm-identity}
 \|df_z\|=|f_z(z)|+|f_{\overline z}(z)|,
\end{equation}
where \(\|df_z\|\) is the operator norm of the real differential \(df_z:\R^2\to\R^2\).
Moreover, its minimal stretching is
\begin{equation}\label{eq:min-stretch-complex}
 \min_{|v|=1}|df_z(v)|
 =\bigl||f_z(z)|-|f_{\overline z}(z)|\bigr|.
\end{equation}
\end{lemma}

\begin{proof}
At the fixed point \(z\), for a unit direction \(e^{i\theta}\),
\[
 df_z(e^{i\theta})=f_z(z)e^{i\theta}+f_{\overline z}(z)e^{-i\theta}.
\]
The triangle inequality gives the upper bound in \eqref{eq:norm-identity}. If both terms are nonzero, we choose \(\theta\) so that they have the same argument; if one term vanishes, equality is immediate.
Choosing \(\theta\) so that the two terms have opposite arguments gives \eqref{eq:min-stretch-complex}; the reverse inequality is the reverse triangle inequality.
\end{proof}

For a complex-valued \(C^1\) map \(f\), we use the notation
\[
    |\nabla f(z)|
    =
    \bigl(|f_x(z)|^2+|f_y(z)|^2\bigr)^{1/2}.
\]
Equivalently,
\[
    \frac{|\nabla f(z)|}{\sqrt2}
    =
    \sqrt{|f_z(z)|^2+|f_{\overline z}(z)|^2}.
\]

We use the following pointwise stretching terminology for Hilbert-valued maps.

\begin{definition}\label{def:pointwise-distortion}
Let \(H\) be a real Hilbert space and let \(F\) be differentiable at a point \(z_0\) of a plane domain. The \emph{maximal stretching} and \emph{minimal stretching} of \(F\) at \(z_0\) are
\begin{equation}\label{eq:max-min-stretching}
 L_F(z_0)=\max_{|v|=1}|dF_{z_0}v|=\|dF_{z_0}\|,
 \qquad
 \ell_F(z_0)=\min_{|v|=1}|dF_{z_0}v|.
\end{equation}
A nonzero differential \(dF_{z_0}\) is conformal if
\[
 |dF_{z_0}v|=\lambda|v|\qquad(v\in\R^2)
\]
for some \(\lambda>0\), equivalently if \(L_F(z_0)=\ell_F(z_0)=\lambda\). For \(K\geq1\), we say that \(F\) is \emph{\(K\)-quasiconformal at \(z_0\) in the pointwise sense} if
\begin{equation}\label{eq:pointwise-qc-def}
 0<\ell_F(z_0)\leq L_F(z_0)\leq K\ell_F(z_0).
\end{equation}
Equivalently, \(dF_{z_0}\) satisfies the pointwise \(K\)-distortion bound. This is a condition only on the differential at \(z_0\); no quasiconformality, injectivity, or homeomorphism property is imposed in a neighborhood.
\end{definition}

For a complex-valued map, Lemma~\ref{lem:operator-norm} gives
\[
 L_f(z)=|f_z(z)|+|f_{\overline z}(z)|,
 \qquad
 \ell_f(z)=\bigl||f_z(z)|-|f_{\overline z}(z)|\bigr|.
\]
Thus Definition~\ref{def:pointwise-distortion} agrees, in the planar orientation-preserving case, with the usual pointwise quasiconformal dilatation. No global quasiconformality will be assumed below.

For \(U\in L^1(\T)\), we write its Poisson integral as
\[
 P[U](re^{i\theta})=\int_\T P_r(\theta-t)U(e^{it})\dm,
 \qquad
 P_r(s)=\frac{1-r^2}{1-2r\cos s+r^2}.
\]
For standard background on harmonic Poisson integrals, see, for example, \cite{AxlerBourdonRamey}.
Differentiation at the origin gives the formulas used below.

\begin{lemma}\label{lem:real-poisson}
Let \(u=P[U]\), where \(U\in L^1(\T)\) is real-valued. Then
\[
 u(0)=\int_\T U\dm,
\]
and
\begin{equation}\label{eq:real-poisson-grad}
 u_x(0)=2\int_\T U(e^{it})\cos t\dm,
 \qquad
 u_y(0)=2\int_\T U(e^{it})\sin t\dm.
\end{equation}
Consequently,
\begin{equation}\label{eq:grad-duality}
 |\nabla u(0)|=2\sup_{\theta\in\R}
 \left|\int_\T U(e^{it})\cos(t-\theta)\dm\right|.
\end{equation}
The absolute value may be omitted if the supremum is taken over all \(\theta\), since replacing \(\theta\) by \(\theta+\pi\) changes the sign.
\end{lemma}

\begin{proof}
The identity at the origin follows from \(P_0\equiv1\). The expansion
\[
 P_r(\theta-t)=1+2r\cos(\theta-t)+O(r^2)
\]
is uniform in \(t\) as \(r\to0\). Its remainder, integrated against
$U$, is bounded by $Cr^2\|U\|_{L^1}$ for a constant $C$ independent of
$U$ and small $r$. Differentiation at the origin therefore gives
\eqref{eq:real-poisson-grad}. Formula \eqref{eq:grad-duality} is the dual characterization of the Euclidean norm of the gradient.
\end{proof}

For a complex-valued boundary function, the value and the two complex
derivatives at the origin are determined by its first Fourier coefficients.

\begin{lemma}\label{lem:complex-poisson}
Let \(f=P[\Phi]\), where \(\Phi\in L^1(\T)\). Then
\begin{equation}\label{eq:complex-poisson-coeff}
 f(0)=\widehat\Phi(0),
 \qquad
 f_z(0)=\widehat\Phi(1),
 \qquad
 f_{\overline z}(0)=\widehat\Phi(-1).
\end{equation}
Moreover,
\[
 f_x(0)=2\int_\T\Phi(e^{it})\cos t\dm,
 \qquad
 f_y(0)=2\int_\T\Phi(e^{it})\sin t\dm.
\]
\end{lemma}

\begin{proof}
We apply Lemma~\ref{lem:real-poisson} to the real and imaginary parts and combine the resulting formulas with the definitions of \(f_z\) and \(f_{\overline z}\). The identities in \eqref{eq:complex-poisson-coeff} agree with the convention \eqref{eq:fourier-convention}.
\end{proof}

We now pass to Hilbert-valued harmonic maps. A \(C^2\) map
\(F:\D\to H\), where \(H\) is a real Hilbert space, is called harmonic
if \(\Delta F=0\). Equivalently, the function
\(z\mapsto\langle F(z),h\rangle\) is harmonic for every \(h\in H\).
For \(\Phi\in L^1(\T,H)\), all integrals below are Bochner integrals,
and \(P[\Phi]\) denotes the Bochner Poisson integral. We use the following
boundary representation for bounded harmonic maps.

\begin{lemma}\label{lem:hilbert-boundary}
Let \(H\) be a real Hilbert space and let \(F:\D\to H\) be bounded and harmonic. Then there is a unique \(\Phi\in L^\infty(\T,H)\), up to equality almost everywhere, such that
\begin{equation}\label{eq:hilbert-poisson-representation}
 F=P[\Phi].
\end{equation}
Moreover,
\begin{equation}\label{eq:hilbert-boundary-norm}
 \|\Phi\|_{L^\infty(\T,H)}=\sup_{z\in\D}|F(z)|.
\end{equation}
If \(C\subset H\) is closed and \(F(\D)\subset C\), then
\begin{equation}\label{eq:hilbert-boundary-convex-target}
 \Phi(\zeta)\in C
 \qquad\text{for almost every }\zeta\in\T.
\end{equation}
\end{lemma}

\begin{proof}
We set \(M=\sup_{\D}|F|\), choose a sequence \(r_j\to1^-\), and define
\[
 \Phi_j(e^{it})=F(r_j e^{it}).
\]
The functions \(\Phi_j\) belong to the set
\[
 \mathcal C_M=\{\Psi\in L^2(\T,H): |\Psi|\leq M\ \text{a.e.}\}.
\]
The set \(\mathcal C_M\) is norm closed and convex, hence weakly closed. Since \(L^2(\T,H)\) is reflexive and \(\mathcal C_M\) is bounded, it is weakly compact. By the Eberlein-\v{S}mulian theorem, after passing to a subsequence we may assume that
\[
 \Phi_j\rightharpoonup\Phi
 \qquad\text{weakly in }L^2(\T,H)
\]
for some \(\Phi\in\mathcal C_M\). In particular, \(\Phi\in L^\infty(\T,H)\) and \(\|\Phi\|_\infty\leq M\).

We fix \(z=\rho e^{i\theta}\in\D\). For all sufficiently large \(j\), the vector-valued Poisson formula on the disk \(|\zeta|<r_j\), obtained by testing against vectors in \(H\), gives
\[
 F(z)=\int_\T P_{\rho/r_j}(\theta-t)\Phi_j(e^{it})\dm.
\]
Since \(P_{\rho/r_j}(\theta-\cdot)\to P_\rho(\theta-\cdot)\) in \(L^1(\T)\) and \(|\Phi_j|\leq M\) almost everywhere,
\[
 \left|F(z)-\int_\T P_\rho(\theta-t)\Phi_j(e^{it})\dm\right|\longrightarrow0.
\]
The map
\[
 \Psi\longmapsto\int_\T P_\rho(\theta-t)\Psi(e^{it})\dm
\]
is a bounded linear operator from \(L^2(\T,H)\) to \(H\). Weak convergence therefore shows that the last integrals converge weakly to \(P[\Phi](z)\). They also converge in norm to \(F(z)\) by the preceding estimate. Hence \(F(z)=P[\Phi](z)\), proving \eqref{eq:hilbert-poisson-representation}.

For uniqueness, suppose that \(P[\Phi]=P[\Psi]\). Since Bochner-measurable functions are essentially separably valued, the essential ranges of \(\Phi\) and \(\Psi\) are contained in a separable closed subspace \(H_0\subset H\). We choose a countable dense subset \(\{h_k\}_{k\geq1}\) of \(H_0\). For every \(k\), the scalar Poisson integral of \(\langle\Phi-\Psi,h_k\rangle\) vanishes, and scalar uniqueness gives
\[
 \langle\Phi-\Psi,h_k\rangle=0
 \qquad\text{a.e. on }\T.
\]
Outside the union of the resulting null sets, \(\Phi-\Psi\in H_0\) is orthogonal to a dense subset of \(H_0\), and hence \(\Phi=\Psi\). Finally, the Poisson estimate gives \(\sup_\D|F|\leq\|\Phi\|_\infty\). Together with \(\|\Phi\|_\infty\leq M=\sup_\D|F|\), this proves \eqref{eq:hilbert-boundary-norm}.

To prove the last assertion, the Poisson approximate identity gives
\[
 \|F(r\,\cdot)-\Phi\|_{L^2(\T,H)}\longrightarrow0
 \qquad(r\to1^-).
\]
Indeed, this convergence is uniform for continuous $H$-valued functions on
$\T$ and extends to $L^2(\T,H)$ by density and the contraction property
of the Poisson operators. We choose $r_j\to1^-$ so that the sum of the
squared $L^2$ errors is finite. Tonelli's theorem then gives
$F(r_j\zeta)\to\Phi(\zeta)$ in $H$ for almost every $\zeta$.
Since each $F(r_j\zeta)$ belongs to the closed set $C$, so does the limit.
This proves \eqref{eq:hilbert-boundary-convex-target} without a convexity
assumption.
\end{proof}

The boundary representation in Lemma~\ref{lem:hilbert-boundary} is sufficient
for the arguments in this paper. A stronger radial-limit statement is also
available.

\begin{remark}\label{rem:hilbert-fatou}
Standard vector-valued Fatou theory gives an almost-everywhere radial-limit
statement for spaces with the Radon-Nikod\'ym property, in particular for
Hilbert spaces; see, for example, \cite{Blasco1988}.
\end{remark}

In particular, the boundary function of every bounded harmonic map
$F:\D\to G\subset H$ belongs to $\overline G$ almost everywhere.
The converse implication need not hold for a nonconvex $G$: the Poisson
integral of an arbitrary $\overline G$-valued function may leave $G$.
We therefore define admissible boundary functions through the harmonic maps
themselves. For the explicit extremals below, membership of the Poisson
integral in the open target is verified directly.

\section{The M\"obius-weighted Hilbert pairing}\label{sec:weighted-support}

We first isolate the algebraic identity used in the directional problem. For \(0\leq s<1\), let
\[
 \mathcal M_s(\zeta)=\frac{s+\zeta}{1+s\zeta},
 \qquad \zeta\in\overline\D,
\]
and, on \(\T\), set
\[
 w_s(z)=|1+sz|^2=1+s^2+2s\re z.
\]
Since \(|z|=1\),
\begin{equation}\label{eq:weighted-auto-general}
 w_s(z)\overline{\mathcal M_s(z)}=z^{-1}+2s+s^2z.
\end{equation}
Thus the weight clears the M\"obius denominator. After integration against a
boundary function, the three terms recover its mean and the Fourier
coefficients that represent the differential at the origin. The planar disk
specialization will be recorded separately in Section~\ref{sec:disk-local}.

We use the following coordinate convention. Let \(H\) be a real Hilbert
space with real-valued inner product \(\langle\cdot,\cdot\rangle\), and let
\(\Lambda\subset H\) be an oriented real two-dimensional subspace with an
oriented orthonormal basis \(e_1,e_2\). We identify
\[
 x_1e_1+x_2e_2\in\Lambda
 \quad\text{with}\quad
 x_1+ix_2\in\C.
\]
Thus complex notation is used only for coordinates in \(\Lambda\), while
every Hilbert-space inner product below is real. If a complex Hilbert
space occurs, we use its underlying real structure and take the real part
of its complex inner product. The differential norm, conformality, and all
inequalities below are invariant under postcomposition by an orthogonal
isometry of \(H\), so this choice of coordinates does not affect any of
the quantities under consideration.

The parameters below are adapted to a round affine disk. Suppose that
\[
 p=p_\Lambda+p_\perp,\qquad
 p_\Lambda=ae_1,\qquad p_\perp\perp\Lambda,
\]
with \(a\geq0\). If a disk in the affine plane \(p_\perp+\Lambda\) is centered
at \(p_\perp\), has radius \(c>a\), and contains \(p\), then \(s=a/c\) is the
M\"obius parameter that sends the origin to the prescribed point. The identity
proved below does not require a target or a round-section hypothesis; hence we
first treat \(c>0\) and \(0\leq s<1\) as free parameters.

Let \(p_\perp\in\Lambda^\perp\), set \(b=|p_\perp|\), and let \(c>0\) and
\(0\leq s<1\). We define
\begin{equation}\label{eq:weighted-E-def}
 E_s^c(\zeta)=p_\perp
 +c\re\mathcal M_s(\zeta)e_1
 +c\operatorname{Im}\mathcal M_s(\zeta)e_2,
 \qquad \zeta\in\overline\D.
\end{equation}
For \(\Phi=(\xi,\eta)\in L^1(\T,H)\), with \(\xi\in L^1(\T,\Lambda)\) and \(\eta\in L^1(\T,\Lambda^\perp)\), write \(\phi=\xi_1+i\xi_2\) under the above identification of \(\Lambda\) with \(\C\). Set
\begin{equation}\label{eq:weighted-total-integral}
 I_s^c(\Phi)
 =\int_\T w_s(z)\langle\Phi(z),E_s^c(z)\rangle\dm.
\end{equation}
This is the \(L^1\)-\(L^\infty\) duality pairing of \(\Phi\) with the
weighted boundary field \(w_sE_s^c\). We call it the
\emph{M\"obius-weighted Hilbert pairing}. The definition involves no target.
All terms are well defined: orthogonal projections preserve Bochner
integrability, and on $\T$ one has
\[
 (1-s)^2\leq w_s\leq(1+s)^2,
 \qquad |E_s^c|=\sqrt{c^2+b^2}.
\]
Consequently,
\[
 |w_s\langle\Phi,E_s^c\rangle|
 \leq(1+s)^2\sqrt{c^2+b^2}\,|\Phi|\in L^1(\T).
\]
We also use $\int_\T w_s\dm=1+s^2$.

The pairing has the following Fourier representation.

\begin{theorem}\label{thm:weighted-pairing-identity}
Assume that \(\Phi=(\xi,\eta)\in L^1(\T,H)\) satisfies
\begin{equation}\label{eq:general-fourier-plane}
 \widehat\phi(0)=a\in\R,
 \qquad
 \widehat\phi(1)=A\in\C,
 \qquad
 \widehat\phi(-1)=B\in\C,
\end{equation}
and
\begin{equation}\label{eq:general-orthogonal-identities}
 \int_\T\eta\dm=p_\perp,
 \qquad
 \int_\T\eta(e^{it})\cos t\dm=0.
\end{equation}
Then
\begin{equation}\label{eq:weighted-pairing-formula}
 I_s^c(\Phi)
 =c\re\bigl(A+2sa+s^2B\bigr)+(1+s^2)b^2.
\end{equation}
\end{theorem}

\begin{proof}
The orthogonal decomposition \(H=\Lambda\oplus\Lambda^\perp\) gives
\[
 I_s^c(\Phi)=I_\Lambda+I_\perp,
\]
where the two terms are the contributions of \(\xi\) and \(\eta\). By \eqref{eq:weighted-auto-general},
\[
\begin{aligned}
 I_\Lambda
 &=\int_\T w_s(z)\langle\xi(z),c\mathcal M_s(z)\rangle\dm\\
 &=c\re\int_\T\phi(z)\bigl(z^{-1}+2s+s^2z\bigr)\dm\\
 &=c\re\bigl(\widehat\phi(1)+2s\widehat\phi(0)+s^2\widehat\phi(-1)\bigr)\\
 &=c\re\bigl(A+2sa+s^2B\bigr).
\end{aligned}
\]
For the normal component, using \(w_s(e^{it})=1+s^2+2s\cos t\), we obtain
\[
\begin{aligned}
 I_\perp
 &=\int_\T w_s(z)\langle\eta(z),p_\perp\rangle\dm\\
 &=(1+s^2)\left\langle\int_\T\eta\dm,p_\perp\right\rangle
   +2s\left\langle\int_\T\eta(e^{it})\cos t\dm,p_\perp\right\rangle\\
 &=(1+s^2)|p_\perp|^2.
\end{aligned}
\]
Combining the two contributions proves \eqref{eq:weighted-pairing-formula}.
\end{proof}

A one-sided bound for the weighted pairing immediately gives a bound for
the conformal factor.

\begin{corollary}\label{cor:pairing-estimate-general}
In the setting of Theorem~\ref{thm:weighted-pairing-identity}, suppose that,
for some \(S\in\R\),
\begin{equation}\label{eq:general-pairing-estimate}
 I_s^c(\Phi)\leq S.
\end{equation}
Then
\begin{equation}\label{eq:pairing-estimate-identity}
 c\re\bigl(A+2sa+s^2B\bigr)+(1+s^2)b^2
 \leq S.
\end{equation}
In particular, if \(A=\lambda>0\) and \(B=0\), then
\begin{equation}\label{eq:general-conformal-bound}
 \lambda\leq
 \frac{S-(1+s^2)b^2-2csa}{c}.
\end{equation}
\end{corollary}

\begin{proof}
Substituting \eqref{eq:weighted-pairing-formula} into
\eqref{eq:general-pairing-estimate} gives \eqref{eq:pairing-estimate-identity}.
If $A=\lambda>0$ and $B=0$, solving the resulting inequality for $\lambda$
gives \eqref{eq:general-conformal-bound}.
\end{proof}

\section{Directional extremals for general targets}
\label{sec:target-geometry}

We apply the identity to harmonic maps with a prescribed value and
a prescribed image plane for their conformal differential.

\subsection{Round sections and the M\"obius extremality criterion}

We consider a bounded domain \(G\) in a real Hilbert space \(H\), a point
\(p\in G\), and a real two-dimensional linear subspace \(\Lambda\subset H\). We
write \(\Pi=p+\Lambda\) for the corresponding affine plane. This notation distinguishes the linear subspace \(\Lambda\) from the
affine plane \(\Pi\).

\begin{definition}\label{def:directional-extremal}
Let \(\mathcal F_G(p,\Lambda)\) be the family of harmonic maps
\(F:\D\to G\) such that \(F(0)=p\), \(dF_0\neq0\) is conformal, and
\[
 dF_0(\R^2)=\Lambda.
\]
The \emph{directional extremal value} of \(G\) at \(p\) in the direction
\(\Lambda\) is \begin{equation}\label{eq:directional-extremal}
 M_G(p,\Lambda)
 =\sup\bigl\{\|dF_0\|:F\in\mathcal F_G(p,\Lambda)\bigr\}.
\end{equation}
\end{definition}

Whenever \(F^0\in\mathcal F_G(p,\Lambda)\) attains the supremum in
\eqref{eq:directional-extremal}, we call \(F^0\) a \emph{directional extremal
harmonic disk} for \((G,p,\Lambda)\). No attainment is assumed in the
definition of \(M_G(p,\Lambda)\).

The family in Definition~\ref{def:directional-extremal} is nonempty: a
sufficiently small affine conformal disk centered at \(p\) lies in \(G\).
The supremum is finite because \(G\) is bounded, Lemma~\ref{lem:hilbert-boundary}
provides an essentially bounded boundary function, and the first-derivative
formulas in Lemma~\ref{lem:real-poisson} apply to its scalar projections. The
image plane \(\Lambda\) remains prescribed in every dimension, and
Definition~\ref{def:directional-extremal} allows maps in
$\mathcal F_G(p,\Lambda)$ to leave the affine plane while remaining in the
ambient target.

Suppose that $G\cap\Pi$ is a Euclidean disk with center $q_0$.
Let $p_\perp$ be the point of $\Pi$ nearest the origin, and put
$t=p_\perp-q_0\in\Lambda$. Translation by $t$ leaves $\Pi$ unchanged.
The correspondence $F\mapsto F+t$ preserves the differential and gives
\begin{equation}\label{eq:intro-translation-invariance}
 M_{G+t}(p+t,\Lambda)=M_G(p,\Lambda).
\end{equation}
The translated section is centered at $q_0+t=p_\perp$. Relabelling the
translated target and prescribed point as $G$ and $p$, we use this
normalization throughout the round-section problem.
Decompose
\begin{equation}\label{eq:directional-decomposition}
 p=p_\Lambda+p_\perp,
 \qquad p_\Lambda\in\Lambda,
 \qquad p_\perp\perp\Lambda.
\end{equation}
Then \(\Pi=p_\perp+\Lambda\), and \(p_\perp\) is the point of \(\Pi\) nearest
the origin. We set \(a=|p_\Lambda|\) and \(b=|p_\perp|\), and assume that, for
some \(c>0\),
\begin{equation}\label{eq:round-affine-section}
 G\cap\Pi
 =D_\Pi(p_\perp,c)
 :=\{p_\perp+v:v\in\Lambda,\ |v|<c\}.
\end{equation}
Since \(p\in G\), one has \(a<c\). We choose an orthonormal basis
\(e_1,e_2\) of \(\Lambda\) with \(p_\Lambda=ae_1\), taking \(e_1\)
arbitrarily when \(a=0\), and set
\begin{equation}\label{eq:directional-s}
 s=\frac ac.
\end{equation}
We also write
\begin{equation}\label{eq:intro-round-parameters}
 R^2=c^2+b^2.
\end{equation}
For $y=p_\perp+cu$ with $u\in\Lambda$ and $|u|=1$, orthogonality
gives $|y|^2=b^2+c^2=R^2$. Thus $c$ is the radius of the affine disk,
and $R$ is the common norm of its boundary points. No containment
$G\subset\B_H(0,R)$ is assumed.
The function \(E_s^c\) defined in \eqref{eq:weighted-E-def} parametrizes
the disk in \eqref{eq:round-affine-section} conformally and satisfies
\(E_s^c(0)=p\). The domain \(G\), the point \(p\), and the plane \(\Lambda\)
determine every parameter: the section fixes \(c\), the orthogonal
decomposition of \(p\) fixes \(a\) and \(b\), and \(s=a/c\). Once an
orientation of \(\Lambda\)
has been chosen, all orientation-preserving conformal diffeomorphisms onto the
section with value \(p\) at the origin are
\begin{equation}\label{eq:mobius-rotation-family}
 T_\theta(z)=E_s^c(e^{i\theta}z),
 \qquad \theta\in\R.
\end{equation}
Indeed, after subtracting \(p_\perp\) and dividing by \(c\), such a
parametrization is an automorphism of \(\D\) with value \(s\) at the origin,
and these automorphisms are precisely
\(\mathcal M_s(e^{i\theta}z)\), \(\theta\in\R\).
Precomposition with a reflection of $\D$ gives the orientation-reversing
parametrizations. For each map in \(\mathcal F_G(p,\Lambda)\), put
\(\lambda=\|dF_0\|>0\). Since \(dF_0/\lambda\) is an isometry from
\(\R^2\) onto \(\Lambda\), one can choose a unit vector in \(\R^2\) whose image
is \(e_1\). A rotation sends the positive \(x\)-axis to this vector; if the
image of the resulting positive \(y\)-axis is \(-e_2\), reflection across the
new \(x\)-axis fixes the first equality and changes the second sign. We thereby
obtain a normalized representative with
\begin{equation}\label{eq:directional-normalization}
 F_x(0)=\lambda e_1,
 \qquad F_y(0)=\lambda e_2,
 \qquad \lambda=\|dF_0\|>0.
\end{equation}
Let \(\mathcal F_G^{\mathrm n}(p,\Lambda)\) denote the members of
\(\mathcal F_G(p,\Lambda)\) satisfying \eqref{eq:directional-normalization}.
If \(\mathcal R\) is the rotation or reflection used above, then
\[
 d(F\circ\mathcal R)_0=dF_0\circ\mathcal R.
\]
Hence precomposition by \(\mathcal R\) preserves harmonicity, the value at the
origin, the target, conformality, the image plane of the differential, and its
operator norm. Every admissible map therefore has a normalized representative,
and
\begin{equation}\label{eq:directional-normalized-equivalence}
 M_G(p,\Lambda)
 =\sup_{F\in\mathcal F_G^{\mathrm n}(p,\Lambda)}\|dF_0\|.
\end{equation}
Let
\begin{equation}\label{eq:directional-boundary-class}
 \mathcal B_G^{\mathrm n}(p,\Lambda)
 =\{\Phi_F:F\in\mathcal F_G^{\mathrm n}(p,\Lambda)\},
\end{equation}
where \(\Phi_F\) is the unique Poisson boundary function of \(F\);
equivalently, it is the radial boundary function almost everywhere. Define
\begin{equation}\label{eq:directional-sigma}
 \sigma_G(p,\Lambda)
 =\sup_{\Phi\in\mathcal B_G^{\mathrm n}(p,\Lambda)}I_s^c(\Phi).
\end{equation}
No absolute value is taken in \eqref{eq:directional-sigma}. The following
equivalence theorem is the analytic reduction for the prescribed conformal
problem: on the normalized family, the pairing is an affine function of
the differential norm with positive slope.

\begin{theorem}
\label{thm:directional-paired-equivalence}
Under assumption \eqref{eq:round-affine-section}, the derivative extremal
problem \eqref{eq:directional-extremal} and the boundary-pairing problem
\eqref{eq:directional-sigma} are equivalent. For every
\(F=P[\Phi_F]\in\mathcal F_G^{\mathrm n}(p,\Lambda)\),
\begin{equation}\label{eq:directional-pairing-value}
 I_s^c(\Phi_F)
 =c\|dF_0\|+2csa+(1+s^2)b^2.
\end{equation}
Consequently,
\begin{equation}\label{eq:directional-primal-paired}
 \sigma_G(p,\Lambda)
 =cM_G(p,\Lambda)+2csa+(1+s^2)b^2.
\end{equation}
The two problems have the same ordering within the normalized class. In
particular, maximizing sequences correspond, and a member of the normalized
class attains one supremum if and only if it attains the other.
\end{theorem}

\begin{proof}
We take \(F=P[\Phi_F]\in\mathcal F_G^{\mathrm n}(p,\Lambda)\) and write
\(\Phi_F=(\xi,\eta)\) according to
\(H=\Lambda\oplus\Lambda^\perp\). Under the identification
\(e_1\leftrightarrow1\), \(e_2\leftrightarrow i\), we denote the complex
coordinate of \(\xi\) by \(\phi\). The value condition and
\eqref{eq:directional-normalization} give
\[
 \widehat\phi(0)=a,
 \qquad \widehat\phi(1)=\lambda,
 \qquad \widehat\phi(-1)=0.
\]
The normal component has mean \(p_\perp\), while
\(F_x(0)\in\Lambda\) gives
\[
 \int_\T\eta(e^{it})\cos t\dm=0.
\]
Theorem~\ref{thm:weighted-pairing-identity} therefore yields
\eqref{eq:directional-pairing-value}.
After taking the supremum over the normalized family and using
\eqref{eq:directional-normalized-equivalence}, we obtain
\eqref{eq:directional-primal-paired}. Subtracting
\eqref{eq:directional-pairing-value} gives
\[
 \sigma_G(p,\Lambda)-I_s^c(\Phi_F)
 =c\bigl(M_G(p,\Lambda)-\|dF_0\|\bigr).
\]
Since $c>0$, one difference vanishes exactly when the other does, and
convergence of either difference to zero along a sequence is equivalent
to convergence of the other. This proves the assertions about extremals
and maximizing sequences.
\end{proof}

The calculation uses only the value and differential conditions.
The round-section hypothesis makes $E_s^c$ admissible and gives the
parameters their geometric meaning; the image of $F$ may leave the section.

Applying the preceding equivalence to the M\"obius parametrization gives a
difference formula and a necessary and sufficient condition for its
extremality.

\begin{corollary}
\label{cor:mobius-extremality}
Assume \eqref{eq:round-affine-section} and retain the normalization above. Let
\[
 T=E_s^c:\D\longrightarrow D_\Pi(p_\perp,c),
 \qquad \Phi_T=T|_{\T},
 \qquad R^2=c^2+b^2.
\]
Then, for every \(F=P[\Phi_F]\in\mathcal F_G^{\mathrm n}(p,\Lambda)\),
\begin{equation}\label{eq:mobius-exact-difference}
 I_s^c(\Phi_F)-(1+s^2)R^2
 =c\bigl(\|dF_0\|-\|dT_0\|\bigr).
\end{equation}
Consequently,
\begin{equation}\label{eq:mobius-iff}
 \|dF_0\|\leq\|dT_0\|
 \quad\Longleftrightarrow\quad
 I_s^c(\Phi_F)\leq(1+s^2)R^2.
\end{equation}
In particular, the following conditions are equivalent:
\begin{enumerate}
 \item[(i)] \(T\) is a directional extremal harmonic disk for
 \((G,p,\Lambda)\);
 \item[(ii)] \(\sigma_G(p,\Lambda)=(1+s^2)R^2\);
 \item[(iii)]
 \(I_s^c(\Phi_F)\leq(1+s^2)R^2\) for every
 \(F\in\mathcal F_G^{\mathrm n}(p,\Lambda)\).
\end{enumerate}
\end{corollary}

\begin{proof}
For \(F\), equation \eqref{eq:directional-pairing-value} gives the value of
\(I_s^c(\Phi_F)\).
For \(T\),
\[
 \|dT_0\|=c(1-s^2),
 \qquad
 I_s^c(\Phi_T)=(1+s^2)R^2,
\]
because \(|\Phi_T|^2=c^2+b^2=R^2\) on \(\T\) and
\(\int_\T w_s\dm=1+s^2\). Applying
\eqref{eq:directional-pairing-value} to \(T\) and subtracting yields
\eqref{eq:mobius-exact-difference}. Equation \eqref{eq:mobius-iff} follows
immediately. Since \(T\) itself belongs to the normalized family,
the three final conditions are equivalent.
\end{proof}

\subsection{Complex targets and holomorphic disks}

We regard $\C^n$ as a real Hilbert space with inner product
\[
 \langle z,w\rangle_{\R}
 =\re\sum_{j=1}^n z_j\overline{w_j}.
\]
A map \(F:\D\to G\subset\C^n\) is harmonic componentwise if and only if
it is harmonic as a map into this real Hilbert space. Lemma~\ref{lem:hilbert-boundary} gives
\begin{equation}\label{eq:complex-boundary-containment}
 \Phi_F(\zeta)\in\overline G\quad\text{almost everywhere},
\end{equation}
without a convexity assumption. Holomorphic disks impose an additional
relation between the two real derivatives.

If $F$ is holomorphic and $F'(0)\ne0$, then $dF_0(\R^2)$ is necessarily a
complex line. Suppose that the affine complex-line section is the round disk
in \eqref{eq:round-affine-section}. We choose the complex orientation of
$\Lambda$ and take the unit generator $u=e_1$ from the preceding
round-section normalization, so that $\Lambda=\C u$ and $e_2=iu$. After
precomposition with a rotation of $\D$, we may normalize
\[
 F'(0)=\lambda u,\qquad \lambda=|F'(0)|>0.
\]
This is exactly the normalization \eqref{eq:directional-normalization}.

For a round complex-line section, the preceding equivalence therefore has the
following holomorphic form.

\begin{corollary}
\label{cor:holomorphic-mobius-criterion}
Let $G\subset\C^n$ be a bounded domain, let $\Lambda$ be a complex line, and assume
\eqref{eq:round-affine-section}. Let $T=E_s^c$ be the corresponding
holomorphic M\"obius parametrization through $p$. For every holomorphic disk
$F:\D\to G$ with $F(0)=p$ and $F'(0)\in\Lambda\setminus\{0\}$, after the
phase normalization above,
\begin{equation}\label{eq:holomorphic-exact-difference}
 I_s^c(\Phi_F)-(1+s^2)R^2
 =c\bigl(|F'(0)|-|T'(0)|\bigr).
\end{equation}
Consequently,
\[
 |F'(0)|\leq|T'(0)|
 \quad\Longleftrightarrow\quad
 I_s^c(\Phi_F)\leq(1+s^2)R^2.
\]
\end{corollary}

\begin{proof}
A holomorphic disk is harmonic and its real differential is conformal with
$\|dF_0\|=|F'(0)|$. We therefore apply
Corollary~\ref{cor:mobius-extremality} with the complex orientation of
$\Lambda$.
\end{proof}

\subsection{A supporting-hyperplane criterion for a round section}

When every point of the boundary circle determines a supporting
hyperplane for the whole target, the integral criterion can be verified
geometrically. The precise sufficient condition is the following.

\begin{proposition}\label{prop:section-support-criterion}
Assume that \(G\subset H\) is a bounded domain satisfying
\eqref{eq:round-affine-section}, and set
\[
 R=\sqrt{c^2+b^2}.
\]
Suppose that every point \(y\) of the boundary circle
\(\partial D_\Pi(p_\perp,c)\) satisfies the supporting-hyperplane condition
\begin{equation}\label{eq:section-support-condition}
 \langle X,y\rangle\leq R^2
 \qquad (X\in\overline G).
\end{equation}
Since \(|y|^2=R^2\), equality holds at \(X=y\); hence the hyperplane
\(\{X\in H:\langle X,y\rangle=R^2\}\) supports \(\overline G\) at \(y\).
Then
\begin{equation}\label{eq:section-support-value}
 \sigma_G(p,\Lambda)=(1+s^2)R^2,
 \qquad
 M_G(p,\Lambda)=\frac{R^2-|p|^2}{c}.
\end{equation}
For $y$ on the boundary circle, let
\[
 C_y=\{X\in\overline G:\langle X,y\rangle=R^2\}.
\]
A map $F=P[\Phi_F]\in\mathcal F_G^{\mathrm n}(p,\Lambda)$ attains this
value if and only if
\begin{equation}\label{eq:section-support-equality}
 \Phi_F(\zeta)\in C_{E_s^c(\zeta)}\quad\text{almost everywhere}.
\end{equation}
If $C_{E_s^c(\zeta)}=\{E_s^c(\zeta)\}$ almost everywhere, the only
normalized extremal is $E_s^c$.
\end{proposition}

\begin{proof}
We take \(F=P[\Phi_F]\in\mathcal F_G^{\mathrm n}(p,\Lambda)\). By Lemma~\ref{lem:hilbert-boundary},
\(\Phi_F(\zeta)\in\overline G\) for almost every \(\zeta\in\T\).
For \(\zeta\in\T\), the point \(E_s^c(\zeta)\) lies on
\(\partial D_\Pi(p_\perp,c)\), and hence
\[
 \langle\Phi_F(\zeta),E_s^c(\zeta)\rangle\leq R^2
 \qquad\text{for almost every }\zeta\in\T.
\]
Since \(\int_\T w_s\dm=1+s^2\), integration gives
\[
 I_s^c(\Phi_F)\leq(1+s^2)R^2.
\]
The conformal parametrization \(E_s^c:\D\to D_\Pi(p_\perp,c)\) belongs to
the normalized family and attains equality, because
\(|E_s^c(\zeta)|^2=c^2+b^2=R^2\) on \(\T\). Thus
\(\sigma_G(p,\Lambda)=(1+s^2)R^2\). Finally, \(cs=a\) and
\(R^2-b^2=c^2\), so Theorem~\ref{thm:directional-paired-equivalence} yields
\[
 M_G(p,\Lambda)
 =\frac{(1+s^2)c^2-2a^2}{c}
 =\frac{c^2-a^2}{c}
 =\frac{R^2-|p|^2}{c}.
\]
Equality is equivalent to the vanishing of the integral of
$w_s(R^2-\langle\Phi_F,E_s^c\rangle)$. Its integrand is nonnegative
and $w_s>0$, so this is precisely \eqref{eq:section-support-equality}.
If the contact sets are singletons, Poisson uniqueness gives $F=E_s^c$.
\end{proof}

The contact sets in the proposition need not be singletons. When $G$ is
convex they are exposed faces of $\overline G$; no convexity is needed
for the equality statement itself.

The supporting condition can also be written as
\[
 c|\pi_\Lambda X|+\langle X,p_\perp\rangle\leq c^2+b^2
 \qquad(X\in\overline G).
\]
Indeed, writing $y=p_\perp+cu$ and taking the supremum over unit vectors
$u\in\Lambda$ gives
\[
 \sup_{\substack{u\in\Lambda\\|u|=1}}
 \langle X,p_\perp+cu\rangle
 =\langle X,p_\perp\rangle+c|\pi_\Lambda X|.
\]

\begin{remark}\label{rem:boundary-containment-support}
The sufficient condition \eqref{eq:section-support-condition} depends on
the choice of origin. It need not hold in fixed coordinates even when
the M\"obius disk is extremal. For example, let
\[
 G=\D\times(-2,2)\subset\C\oplus\R,\qquad
 \Lambda=\C\oplus\{0\},\qquad p=(a,1),\quad0\leq a<1.
\]
Then $G\cap(p+\Lambda)=\D\times\{1\}$, $c=b=1$, and $R^2=2$.
For $F=(f,V)\in\mathcal F_G(p,\Lambda)$, one has $dV_0=0$ and
$df_0$ is conformal, possibly with reversed orientation. The disk estimate
of Section~\ref{sec:disk-local} therefore gives
$\|dF_0\|\leq1-a^2$, with equality for $T=(\mathcal M_a,1)$.
However, $y=(1,1)$ is on the boundary circle and
$X=(3/4,3/2)\in G$ satisfies $\langle X,y\rangle=9/4>R^2$.
For $0<\varepsilon<1$, the maps
\[
 F_\varepsilon(z)=
 \bigl(\mathcal M_a(z),1+\varepsilon\re(z^2)\bigr)
\]
also attain the extremal value and leave the affine section. Translation
by $(0,-1)$ makes this section central, and the supporting condition
then holds. Thus the example concerns the condition in fixed coordinates,
not its validity after every possible translation.
\end{remark}

\subsection{Non-ball examples and rigidity of planar sections}

A round affine section by itself does not determine the directional extremal
value; the ambient target can admit harmonic disks that leave the section and
have a larger differential.

\begin{example}\label{ex:round-section-not-sufficient}
Let
\[
 G=\bigl\{(z,w)\in\C^2:
 |z|^2<1+4|w|,\quad |z|^2+|w|^2<6\bigr\},
\]
regarded as a subset of the underlying real Hilbert space. It is open and
bounded. It is also star-shaped with respect to the origin: if
$(z,w)\in G$ and $0<t\leq1$, then
\[
 t^2|z|^2<t^2(1+4|w|)\leq1+4t|w|
\]
and
\[
 t^2(|z|^2+|w|^2)<6,
\]
while $0\in G$. Hence $G$ is connected and is a domain. We take
\[
 p=0,\qquad \Lambda=\C\times\{0\}.
\]
Then
\[
 G\cap\Lambda=\D\times\{0\},
\]
so the normalized M\"obius parametrization of the section is
\[
 T(\zeta)=(\zeta,0),\qquad \|dT_0\|=1.
\]
However,
\[
 F(\zeta)=(2\zeta,\zeta^2)
\]
is a holomorphic, hence harmonic, disk in $G$. Indeed, for $|\zeta|<1$,
\[
 |2\zeta|^2=4|\zeta|^2<1+4|\zeta|^2
 =1+4|\zeta^2|
\]
and
\[
 |2\zeta|^2+|\zeta^2|^2
 =4|\zeta|^2+|\zeta|^4<5<6.
\]
Moreover, $F(0)=0$, the differential $dF_0$ is conformal with image plane
$\Lambda$, and $\|dF_0\|=2$. Thus $T$ is not extremal. In particular, the
roundness of $G\cap(p+\Lambda)$ alone does not control the extremal problem;
information about the ambient target is essential.
\end{example}

The supporting-hyperplane condition, on the other hand, does not characterize
balls. The following rotational construction gives a bounded convex non-ball
target for which one round affine section satisfies the condition and its
M\"obius parametrization is directionally extremal.

\begin{example}\label{ex:rotational-section-support}
Let \(z_0\in(0,1)\) and set
\[
 c=\sqrt{1-z_0^2}.
\]
In the \(xz\)-plane, let \(A=(c,z_0)\) be the point on the upper unit
semicircle at height \(z_0\). The tangent line at \(A\) is
\[
 cx+z_0z=1,
\]
and it meets the \(x\)-axis at \(A_0=(c^{-1},0)\). Rotating the planar profile formed by the circular arc from \(A\) to
\((0,1)\) and the segment \(AA_0\) about the \(z\)-axis gives a surface
of revolution. This surface and its base disk in the plane \(z=0\)
bound a solid of revolution. Write
\[
 e_3=(0,0,1),\qquad
 \Lambda=\R^2\times\{0\}.
\]
Then the interior of this solid in \(\{z>0\}\) is the bounded rotational
domain
\begin{equation}\label{eq:rotational-domain}
 G_{z_0}
 =\bigl\{v+ze_3:\ v\in\Lambda,\ 0<z<1,\ |v|<R_{z_0}(z)\bigr\},
\end{equation}
where
\begin{equation}\label{eq:rotational-radius}
 R_{z_0}(z)=
 \begin{cases}
 \dfrac{1-z_0z}{c},&0\leq z\leq z_0,\\[6pt]
 \sqrt{1-z^2},&z_0\leq z<1.
 \end{cases}
\end{equation}
The two pieces have the same value and derivative at \(z_0\), and
\(R_{z_0}\) is concave. Indeed, if
\(X_j=v_j+z_je_3\in G_{z_0}\) and \(0<\theta<1\), then
\[
\begin{aligned}
 |\theta v_1+(1-\theta)v_2|
 &\leq\theta|v_1|+(1-\theta)|v_2|\\
 &<\theta R_{z_0}(z_1)+(1-\theta)R_{z_0}(z_2)\\
 &\leq R_{z_0}\bigl(\theta z_1+(1-\theta)z_2\bigr).
\end{aligned}
\]
Thus \(G_{z_0}\) is convex. Since the first
branch in \eqref{eq:rotational-radius} is the tangent line to
\(z\mapsto\sqrt{1-z^2}\) at \(z_0\), the domain \(G_{z_0}\) contains the open upper half-ball
\(\B_{\R^3}\cap\{z>0\}\) and is strictly larger below the level
\(z=z_0\). On the band $z_0<z<1$, however, its boundary agrees with the
unit sphere. If $G_{z_0}$ were a ball, the two spheres would therefore agree,
and $G_{z_0}$ would be the unit ball, contrary to the strict inclusion below
$z=z_0$. Hence $G_{z_0}$ is not a ball.

Nevertheless, for the affine plane
\[
 \Pi=z_0e_3+\Lambda
\]
one has
\begin{equation}\label{eq:rotational-round-section}
 G_{z_0}\cap\Pi
 =\{z_0e_3+v:\ v\in\Lambda,\ |v|<c\},
\end{equation}
so the section is a Euclidean disk. If
\[
 q=z_0e_3+cu,\qquad u\in\Lambda,\quad |u|=1,
\]
then
\begin{equation}\label{eq:rotational-support}
 \langle X,q\rangle\leq1
 \qquad (X\in\overline{G_{z_0}}).
\end{equation}
Indeed, write \(X=v+ze_3\), with \(v\in\Lambda\). For
\(0\leq z\leq z_0\),
\[
 \langle X,q\rangle
 \leq c|v|+z_0z
 \leq c\frac{1-z_0z}{c}+z_0z=1,
\]
whereas for \(z_0\leq z\leq1\),
\[
 \langle X,q\rangle
 \leq c\sqrt{1-z^2}+z_0z
 \leq\sqrt{c^2+z_0^2}\,
       \sqrt{1-z^2+z^2}=1.
\]

For \(p\in G_{z_0}\cap\Pi\), write
\[
 p=p_\Lambda+z_0e_3,\qquad
 a=|p_\Lambda|<c,
\]
and choose an orthonormal basis \(e_1,e_2\) of \(\Lambda\) with
\(p_\Lambda=ae_1\), taking \(e_1\) arbitrarily when \(a=0\). We set
\[
 s=\frac ac.
\]
Then the boundary restriction of \(E_s^c\) from \eqref{eq:weighted-E-def}, with
\(p_\perp=z_0e_3\), parametrizes the circle in
\eqref{eq:rotational-round-section}. Since \(c^2+z_0^2=1\),
\eqref{eq:rotational-support} is precisely the supporting-hyperplane condition in
Proposition~\ref{prop:section-support-criterion}, with \(R=1\).
Consequently,
\begin{equation}\label{eq:rotational-sigma}
 \sigma_{G_{z_0}}(p,\Lambda)=1+s^2.
\end{equation}
The same proposition gives
\begin{equation}\label{eq:rotational-directional-value}
 M_{G_{z_0}}(p,\Lambda)
 =\frac{c^2-a^2}{c}
 =\frac{1-|p|^2}{\sqrt{1-z_0^2}}.
\end{equation}
This is exactly the value of
\(M_{\B_{\R^3}}(p,\Lambda)\) for the unit ball. Thus a non-ball target can
have a round affine section that satisfies the same supporting-hyperplane
condition and has the same directional extremal value as the corresponding
section of the unit ball.
\end{example}

We now consider what happens when every central real two-dimensional
section is a disk.

\begin{theorem}\label{thm:circular-section-rigidity}
Let \(H\) be a real Hilbert space with \(\dim H\geq2\), and let \(G\subset H\)
be an open set containing the origin. Suppose that for every real
two-dimensional linear subspace \(\Lambda\subset H\) there is a number
\(r_\Lambda>0\) such that
\begin{equation}\label{eq:central-circular-sections}
 G\cap\Lambda
 =\{x\in\Lambda:|x|<r_\Lambda\}.
\end{equation}
Then all the radii \(r_\Lambda\) are equal and \(G\) is an open Hilbert ball
centered at the origin.
\end{theorem}

\begin{proof}
For a unit vector \(u\in H\), we set
\[
 \rho(u)=\sup\{t>0:tu\in G\}.
\]
Assumption \eqref{eq:central-circular-sections} shows that \(\rho(u)\) equals
\(r_\Lambda\) for every real two-plane \(\Lambda\) containing \(u\). If
\(u\) and \(v\) are linearly independent unit vectors,
then both belong to the real plane \(\operatorname{span}\{u,v\}\), and
\eqref{eq:central-circular-sections} gives
\(\rho(u)=\rho(v)=r_{\operatorname{span}\{u,v\}}\). The same equality is
immediate when \(u\) and \(v\) are linearly dependent. Hence \(\rho\) is a
constant, say \(R\), on the unit sphere. Every nonzero vector belongs to a
real two-dimensional subspace, so \eqref{eq:central-circular-sections} now
gives
\[
 x\in G\quad\Longleftrightarrow\quad |x|<R.
\]
Thus \(G=\B_H(0,R)\).
\end{proof}

The use of real two-planes in
Theorem~\ref{thm:circular-section-rigidity} is essential. The following
example shows that circular complex-line sections need not force a ball, even
for a balanced convex domain in a complex Hilbert space.

\begin{example}\label{ex:anisotropic-complex-ellipsoid}
Consider
\begin{equation}\label{eq:anisotropic-complex-ellipsoid}
 \mathcal E
 =\bigl\{(z_1,z_2)\in\C^2:
 |z_1|^2+|z_2|^4<1\bigr\}.
\end{equation}
The domain \(\mathcal E\) is bounded and convex. It is also balanced: if
\(z\in\mathcal E\) and \(|\zeta|\leq1\), then \(\zeta z\in\mathcal E\). If
\(v=(v_1,v_2)\neq0\), then
\[
 \zeta v\in\mathcal E
 \quad\Longleftrightarrow\quad
 |\zeta|^2|v_1|^2+|\zeta|^4|v_2|^4<1.
\]
The right-hand side depends only on \(|\zeta|\), so the intersection of
\(\mathcal E\) with every complex line through the origin is a Euclidean disk
in that line. On the real plane
\[
 \Pi=\{(x,y):x,y\in\R\}\subset\C^2,
\]
however, the section is
\[
 \mathcal E\cap\Pi=\{(x,y):x^2+y^4<1\},
\]
which is not a Euclidean disk. Indeed, its radial boundary points on both
coordinate axes have distance \(1\) from the origin, whereas
\((2^{-1/2},2^{-1/2})\), also of Euclidean norm \(1\), lies in the interior.
This example distinguishes complex balance from isotropy in the underlying
real Hilbert geometry. More generally, the intersection of any balanced
domain in \(\C^n\) containing the origin with a complex line through the
origin is a centered disk in that line, possibly the whole line in the
unbounded case. Thus central complex-line circularity is much weaker than
isotropy of real two-dimensional sections.
\end{example}

For \(\mathcal E\), the supporting-hyperplane condition in
Proposition~\ref{prop:section-support-criterion} can be tested explicitly
on each central complex-line section. We now determine exactly the lines
for which this condition holds.

\begin{proposition}\label{prop:ellipsoid-support-lines}
Regard \(\C^2\) as a real Hilbert space and let \(\mathcal E\) be the domain
in \eqref{eq:anisotropic-complex-ellipsoid}. Let
\(\Lambda=\C v\), where \(v=(v_1,v_2)\) and \(|v|=1\), and let
\(r_\Lambda>0\) be determined by
\[
 \mathcal E\cap\Lambda
 =\{\zeta v:|\zeta|<r_\Lambda\}.
\]
Then the boundary circle of this section satisfies
\begin{equation}\label{eq:ellipsoid-line-support}
 \langle X,q\rangle\leq r_\Lambda^2
 \qquad
 \bigl(X\in\overline{\mathcal E},\quad
 q\in\partial(\mathcal E\cap\Lambda)\bigr)
\end{equation}
if and only if either
\begin{equation}\label{eq:ellipsoid-coordinate-directions}
 v_1v_2=0,
\end{equation}
or
\begin{equation}\label{eq:ellipsoid-mixed-direction}
 |v_1|^2=\frac35,
 \qquad
 |v_2|^2=\frac25.
\end{equation}
In these directions, for every \(p\in\mathcal E\cap\Lambda\), the M\"obius
parametrization of \(\mathcal E\cap\Lambda\) through \(p\) is a directional
extremal harmonic disk and
\begin{equation}\label{eq:ellipsoid-directional-value}
 M_{\mathcal E}(p,\Lambda)
 =\frac{r_\Lambda^2-|p|^2}{r_\Lambda}.
\end{equation}
For the two coordinate complex lines, \(r_\Lambda=1\); in the mixed case
\eqref{eq:ellipsoid-mixed-direction}, \(r_\Lambda=\sqrt5/2\).
\end{proposition}

\begin{proof}
We set
\[
 \rho(z_1,z_2)=|z_1|^2+|z_2|^4,
 \qquad
 \mathcal E=\{\rho<1\}.
\]
The function \(\rho\) is convex and the boundary of \(\mathcal E\) is smooth.
Since \(\mathcal E\) is invariant under multiplication by unimodular
scalars, it is enough to test \eqref{eq:ellipsoid-line-support} at
\(q=r_\Lambda v\). Since \(|q|=r_\Lambda\), the hyperplane
\[
 \{X:\langle X,q\rangle=r_\Lambda^2\}
\]
passes through \(q\) and has normal vector \(q\). Because \(\mathcal E\) is
convex and its boundary is smooth, this hyperplane supports \(\mathcal E\)
at \(q\) if and only if the outward normal at \(q\) is a positive multiple of \(q\).
With respect to the underlying real inner product,
\[
 D\rho_q(h)
 =\left\langle h,
 \bigl(2q_1,4|q_2|^2q_2\bigr)\right\rangle,
\]
so the outward normal direction is
\[
 \bigl(2q_1,4|q_2|^2q_2\bigr).
\]
If \(q_1q_2\neq0\), this positive proportionality forces
\(4|q_2|^2=2\), hence \(|q_2|^2=1/2\). Since \(q\in\partial\mathcal E\),
\[
 |q_1|^2+|q_2|^4=1,
\]
and therefore \(|q_1|^2=3/4\). Consequently
\(|q|^2=5/4\), and division by
\(r_\Lambda^2=|q|^2\) gives
\eqref{eq:ellipsoid-mixed-direction}; it also gives
\(r_\Lambda=\sqrt5/2\). If one coordinate of \(q\) vanishes, the outward
normal is automatically a positive multiple of \(q\), giving precisely the coordinate
directions \eqref{eq:ellipsoid-coordinate-directions}. The converse follows
from the same normal calculation, and smooth convexity then gives the
supporting inequality for every point of the boundary circle.

For the stated directions, Proposition~\ref{prop:section-support-criterion}
applies with \(b=0\), \(c=R=r_\Lambda\), and yields
\eqref{eq:ellipsoid-directional-value}. Corollary~\ref{cor:mobius-extremality}
then identifies the corresponding M\"obius parametrization as a directional
extremal harmonic disk.
\end{proof}

The preceding proposition identifies exactly the central complex lines of
\(\mathcal E\) on which the supporting-hyperplane criterion applies.
Extremality, however, holds on every central complex line of a bounded
balanced convex domain.

\begin{theorem}\label{thm:balanced-complex-lines}
Let $G\subset\C^n$ be a bounded balanced convex domain containing the
origin. For a unit vector $v$, put $\Lambda=\C v$ and let $r_\Lambda>0$
be determined by
\[
 G\cap\Lambda=\{\zeta v:|\zeta|<r_\Lambda\}.
\]
Then, for every $p\in G\cap\Lambda$,
\begin{equation}\label{eq:balanced-complex-line-value}
 M_G(p,\Lambda)=\frac{r_\Lambda^2-|p|^2}{r_\Lambda}.
\end{equation}
The M\"obius parametrization of $G\cap\Lambda$ through $p$ is a directional
extremal harmonic disk.
\end{theorem}

\begin{proof}
We write
\[
 \mu_G(z)=\inf\{t>0:z\in tG\}
\]
for the Minkowski functional of $G$. Boundedness, balance, and convexity
make $\mu_G$ a complex norm, with $G=\{z:\mu_G(z)<1\}$ and
$\mu_G(\zeta v)=|\zeta|/r_\Lambda$. The complex Hahn-Banach theorem
extends the functional $\zeta v\mapsto\zeta/r_\Lambda$ to a complex
linear functional $\mathcal L_v:\C^n\to\C$ satisfying
\[
 |\mathcal L_v(z)|\leq\mu_G(z),
 \qquad \mathcal L_v(\zeta v)=\frac{\zeta}{r_\Lambda}.
\]
In particular, $\mathcal L_v(G)\subset\D$.
For $F\in\mathcal F_G(p,\Lambda)$, the harmonic map
$f=\mathcal L_v\circ F:\D\to\D$ has a nonzero conformal differential
at the origin. Since $\mathcal L_v|_\Lambda$ is a similarity of ratio
$1/r_\Lambda$,
\[
 \|df_0\|=\frac{\|dF_0\|}{r_\Lambda},
 \qquad |f(0)|=\frac{|p|}{r_\Lambda}.
\]
The unit-disk case of Proposition~\ref{prop:section-support-criterion}
therefore gives
\[
 \frac{\|dF_0\|}{r_\Lambda}
 \leq1-\frac{|p|^2}{r_\Lambda^2}.
\]
The M\"obius parametrization of $G\cap\Lambda$ through $p$ attains this
bound, which proves \eqref{eq:balanced-complex-line-value}.
\end{proof}

The theorem does not assert uniqueness for a general balanced convex target.

For example, $v=(1,1)/\sqrt2$ gives a central complex-line section of
$\mathcal E$ with
\[
 \frac{r_\Lambda^2}{2}+\frac{r_\Lambda^4}{4}=1,
 \qquad r_\Lambda^2=\sqrt5-1.
\]
This direction satisfies neither \eqref{eq:ellipsoid-coordinate-directions}
nor \eqref{eq:ellipsoid-mixed-direction}, so the supporting-hyperplane
condition \eqref{eq:ellipsoid-line-support} fails there. Its M\"obius disk
is nevertheless extremal by Theorem~\ref{thm:balanced-complex-lines}.
Thus Proposition~\ref{prop:ellipsoid-support-lines} describes exactly the
central complex lines on which the supporting-hyperplane criterion applies,
whereas Theorem~\ref{thm:balanced-complex-lines} establishes extremality on
all central complex lines.

We next write the disk estimate used above directly in terms of the
relevant Fourier coefficients.

\section{The disk-valued problem with conformality at the reference point}\label{sec:disk-local}

The planar disk problem is the scalar specialization of the weighted
identity from Section~\ref{sec:weighted-support}, and in this setting the
boundary calculation can be written explicitly. At the origin we impose the complex-linearity condition
\(f_{\overline z}(0)=0\). When \(df_0\ne0\), this is precisely
orientation-preserving conformality at the origin; allowing \(df_0=0\) does not
change the extremal value. No conformality in a neighborhood is required. By
Lemma~\ref{lem:complex-poisson}, the condition
\(f_{\overline z}(0)=0\) is exactly \(\widehat\Phi(-1)=0\).

For \(p\in\D\), define
\[
 M_{\mathrm{conf}}(p)=\sup\{ |f_z(0)|:\ f:\D\to\D\text{ harmonic},\ f(0)=p,
 f_{\overline z}(0)=0\}.
\]
By multiplying the range by a unimodular constant, it is enough to consider \(p=r\in[0,1)\). Lemmas~\ref{lem:complex-poisson} and~\ref{lem:hilbert-boundary}, applied with \(H=\R^2\), give the boundary form
\begin{equation}\label{eq:disk-boundary-abs}
 M_{\mathrm{conf}}(r)=\sup\{ |\widehat\Phi(1)|:\ |\Phi|\leq1,
 \widehat\Phi(0)=r,
 \widehat\Phi(-1)=0\}.
\end{equation}
Indeed, every harmonic map occurring in the definition of
$M_{\mathrm{conf}}(r)$ has such a boundary function. Conversely, suppose that
$|\Phi|\leq1$ almost everywhere and $\widehat\Phi(0)=r<1$, and set
$u=P[\Phi]$. If $|u(z_0)|=1$ at an interior point
$z_0=\rho e^{i\theta}$, put $\omega=u(z_0)$. Since
$1-\re(\Phi(e^{it})\overline\omega)\geq0$ almost everywhere and the Poisson
kernel is strictly positive,
\[
 0=1-\re(u(z_0)\overline\omega)
 =\int_{\T}P_\rho(\theta-t)
 \bigl(1-\re(\Phi(e^{it})\overline\omega)\bigr)\dm
\]
implies $\re(\Phi(e^{it})\overline\omega)=1$ almost everywhere. Together with
$|\Phi|\leq1=|\omega|$, this gives $\Phi=\omega$ almost everywhere, contrary to
$|\widehat\Phi(0)|=r<1$. Hence $P[\Phi]$ takes values in $\D$, and the
boundary formulation in \eqref{eq:disk-boundary-abs} introduces no
additional maps.
The phase of \(\widehat\Phi(1)\) may be normalized by rotating the boundary
variable. Hence \eqref{eq:disk-boundary-abs} is equivalent to
\begin{equation}\label{eq:disk-boundary-real}
 M_{\mathrm{conf}}(r)=\sup\{ \re\widehat\Phi(1):\ |\Phi|\leq1,
 \widehat\Phi(0)=r,
 \widehat\Phi(-1)=0\}.
\end{equation}
Indeed, if \(\Psi(e^{it})=\Phi(e^{i(t+\theta)})\), then \(\widehat\Psi(n)=e^{in\theta}\widehat\Phi(n)\). The vanishing of \(\widehat\Phi(-1)\) is preserved by this rotation, but it cannot be created by it.

Using only the \(L^2\) consequence of the disk constraint gives a weaker
estimate. If \(|\Phi|\leq1\), \(\widehat\Phi(0)=r\), and \(\widehat\Phi(-1)=0\), then Bessel's inequality gives
\begin{equation}\label{eq:parseval-bound}
 |\widehat\Phi(1)|^2\leq
 \sum_{n\neq0}|\widehat\Phi(n)|^2
 \leq \int_\T|\Phi|^2\dm-r^2\leq1-r^2.
\end{equation}
Thus \(|\widehat\Phi(1)|\leq\sqrt{1-r^2}\). The relaxed problem with
\(\|\Phi\|_2\leq1\) attains this value at
$\Phi(z)=r+\sqrt{1-r^2}\,z$. For $0<r<1$, however, this function violates
the pointwise constraint, since its maximum modulus on $\T$ is
$r+\sqrt{1-r^2}>1$. The weighted argument below gives the stronger
bound $1-r^2$ for the original problem.

To specialize Theorem~\ref{thm:weighted-pairing-identity} to the
disk-valued problem, we set
\begin{equation}\label{eq:q-def}
 q_r(z)=z^{-1}+2r+r^2z=z^{-1}(1+rz)^2,
 \qquad z\in\T,\quad 0\leq r<1.
\end{equation}
Identity~\eqref{eq:weighted-auto-general} gives
$q_r=w_r\overline{\mathcal M_r}$, and hence
\begin{equation}\label{eq:q-modulus}
 |q_r(z)|=|1+rz|^2=1+r^2+2r\re z.
\end{equation}
Thus $q_r$ has no zeros on $\T$, and the disk constraint gives the
following pointwise inequality.

\begin{lemma}\label{lem:disk-support}
For every \(z\in\T\) and every \(\zeta\in\overline\D\),
\[
 \re(\zeta q_r(z))\leq |q_r(z)|.
\]
Equality holds if and only if \(|\zeta|=1\) and \(\zeta q_r(z)\) is a nonnegative real number.
\end{lemma}

\begin{proof}
Since \(|\zeta|\leq1\),
\[
 \re\bigl(\zeta q_r(z)\bigr)
 \leq |\zeta q_r(z)|
 =|\zeta|\,|q_r(z)|
 \leq |q_r(z)|.
\]
Since \(q_r(z)\neq0\) on \(\T\), equality holds precisely when
\(|\zeta|=1\) and \(\zeta q_r(z)\) is a nonnegative real number.
\end{proof}

Integrating the pointwise inequality gives the required coefficient bound.

\begin{proposition}\label{prop:disk-coeff-real}
Let \(0\leq r<1\). If \(\Phi\in L^\infty(\T)\), \(|\Phi|\leq1\), \(\widehat\Phi(0)=r\), and \(\widehat\Phi(-1)=0\), then
\[
 \re\widehat\Phi(1)\leq1-r^2.
\]
The constant is sharp.
\end{proposition}

\begin{proof}
Since $q_r=w_r\overline{\mathcal M_r}$ on $\T$,
Lemma~\ref{lem:disk-support} and \eqref{eq:q-modulus} give
\[
 I_r^1(\Phi)=\re\int_\T\Phi(z)q_r(z)\dm
 \leq\int_\T|q_r(z)|\dm=1+r^2.
\]
We apply Theorem~\ref{thm:weighted-pairing-identity} with $H=\R^2$,
$c=1$, $b=0$, $a=s=r$, $A=\widehat\Phi(1)$, and $B=0$. This gives
$I_r^1(\Phi)=\re\widehat\Phi(1)+2r^2$, and hence
$\re\widehat\Phi(1)\leq1-r^2$.

The boundary function
\[
 \mathcal M_r(z)=\frac{r+z}{1+rz}
\]
is unimodular on \(\T\), and
\[
 \mathcal M_r(z)=r+(1-r^2)z-r(1-r^2)z^2+r^2(1-r^2)z^3-\cdots.
\]
Thus \(\widehat{\mathcal M_r}(0)=r\), \(\widehat{\mathcal M_r}(-1)=0\), and \(\widehat{\mathcal M_r}(1)=1-r^2\). The constant is attained.
\end{proof}

The equality condition in Lemma~\ref{lem:disk-support} also determines the
boundary function.

\begin{proposition}\label{prop:disk-equality-real}
Equality in Proposition~\ref{prop:disk-coeff-real} holds if and only if
\[
 \Phi(z)=\frac{r+z}{1+rz}
\]
for almost every \(z\in\T\).
\end{proposition}

\begin{proof}
If equality holds after integration, the nonnegative function
$|q_r|-\re(\Phi q_r)$ has integral zero and therefore vanishes almost
everywhere. Equality thus holds in Lemma~\ref{lem:disk-support} almost
everywhere. Hence \(\Phi(z)q_r(z)=|q_r(z)|\) a.e., and so
\[
 \Phi(z)=\frac{\overline{q_r(z)}}{|q_r(z)|}.
\]
From \(q_r(z)=z^{-1}(1+rz)^2\) and \(|z|=1\), we get
\[
 \frac{\overline{q_r(z)}}{|q_r(z)|}
 =\frac{z(1+rz^{-1})^2}{(1+rz)(1+rz^{-1})}
 =\frac{r+z}{1+rz}.
\]
The converse follows from the extremal function already exhibited.
\end{proof}

Rotating the boundary variable gives the absolute-value form.

\begin{corollary}\label{cor:disk-coeff-abs}
Let \(0\leq r<1\). If \(\Phi\in L^\infty(\T)\), \(|\Phi|\leq1\), \(\widehat\Phi(0)=r\), and \(\widehat\Phi(-1)=0\), then
\[
 |\widehat\Phi(1)|\leq1-r^2.
\]
Equality holds if and only if, for some \(\theta\in\R\),
\[
 \Phi(e^{it})=\frac{r+e^{i\theta}e^{it}}{1+r e^{i\theta}e^{it}}
\]
for almost every \(t\).
\end{corollary}

\begin{proof}
We choose \(\theta\) so that \(e^{i\theta}\widehat\Phi(1)=|\widehat\Phi(1)|\) and set \(\Psi(e^{it})=\Phi(e^{i(t+\theta)})\). Then \(\widehat\Psi(0)=r\), \(\widehat\Psi(-1)=0\), and \(\re\widehat\Psi(1)=|\widehat\Phi(1)|\). Applying Propositions~\ref{prop:disk-coeff-real} and~\ref{prop:disk-equality-real} and then undoing the rotation gives the result.
Conversely, the displayed boundary function has
$\widehat\Phi(1)=e^{i\theta}(1-r^2)$ and attains equality.
\end{proof}

\begin{theorem}\label{thm:invariant-coeff}
Let \(\Phi\in L^\infty(\T)\), \(|\Phi|\leq1\), and \(\widehat\Phi(-1)=0\). Then
\begin{equation}\label{eq:invariant-coeff}
 |\widehat\Phi(1)|\leq1-|\widehat\Phi(0)|^2.
\end{equation}
If equality holds and \(p=\widehat\Phi(0)\), then either \(|p|=1\) and
\(\Phi=p\) almost everywhere, or \(|p|<1\) and, up to rotations of the range
and of the boundary variable, \(\Phi\) is the boundary value of the automorphism
\((r+z)/(1+rz)\), where \(r=|p|\).
\end{theorem}

\begin{proof}
If \(|p|=1\), equality in
\[
 |p|=\left|\int_\T\Phi\dm\right|
 \leq\int_\T|\Phi|\dm\leq1
\]
forces \(\Phi=p\) almost everywhere. Hence all its Fourier coefficients with nonzero index vanish, and \eqref{eq:invariant-coeff} holds with equality.
We now assume that \(|p|<1\). We choose \(\theta\) so that
\(e^{-i\theta}p=r=|p|\) and apply Corollary~\ref{cor:disk-coeff-abs} to
\(e^{-i\theta}\Phi\). Multiplication by \(e^{-i\theta}\) preserves the condition \(\widehat\Phi(-1)=0\) and does not change \(|\widehat\Phi(1)|\). The equality statement follows in the same way.
\end{proof}

Theorem~\ref{thm:invariant-coeff} is a boundary version of the first Schur
coefficient estimate. Full analyticity is not assumed; the coefficients
\(\widehat\Phi(-2),\widehat\Phi(-3),\ldots\) are unrestricted.
Nevertheless, the single condition \(\widehat\Phi(-1)=0\) gives for
\(\widehat\Phi(1)\) the same best possible bound as full analyticity. When
\(|\widehat\Phi(0)|<1\), equality is stronger than the hypothesis: it forces
the boundary function to be automorphic.

\begin{theorem}\label{thm:main-disk}
For every \(p\in\D\),
\[
 M_{\mathrm{conf}}(p)=1-|p|^2.
\]
More precisely, if \(f:\D\to\D\) is harmonic, \(f(0)=p\), and \(f_{\overline z}(0)=0\), then
\[
 |f_z(0)|\leq1-|p|^2.
\]
Equality holds if and only if \(f\) is a holomorphic automorphism of \(\D\) with \(f(0)=p\).
\end{theorem}

\begin{proof}
By Lemmas~\ref{lem:complex-poisson} and~\ref{lem:hilbert-boundary},
$f=P[\Phi]$ with $|\Phi|\leq1$, $\widehat\Phi(0)=p$,
$\widehat\Phi(-1)=0$, and $f_z(0)=\widehat\Phi(1)$.
Theorem~\ref{thm:invariant-coeff} gives the estimate. If equality holds,
that theorem identifies $\Phi$ as the boundary value of a holomorphic
automorphism; Poisson uniqueness then identifies $f$ with that
automorphism. Conversely, every holomorphic automorphism with value $p$
at the origin has $|f_z(0)|=1-|p|^2$, so the supremum is attained.
\end{proof}

Precomposition with automorphisms of $\D$ gives the pointwise form.

\begin{corollary}\label{cor:disk-pointwise}
Let \(f:\D\to\D\) be harmonic. If \(f_{\overline z}(z_0)=0\), then
\[
 |f_z(z_0)|\leq\frac{1-|f(z_0)|^2}{1-|z_0|^2}.
\]
The estimate is sharp for every \(z_0\in\D\) and every prescribed value \(f(z_0)\in\D\).
Equality holds if and only if $f$ is a holomorphic automorphism of $\D$.
\end{corollary}

\begin{proof}
We choose \(\varphi\in\Aut(\D)\) with \(\varphi(0)=z_0\). Then \(|\varphi'(0)|=1-|z_0|^2\). The map \(F=f\circ\varphi\) is harmonic, \(F(0)=f(z_0)\), and
\[
 F_{\overline z}(0)=f_{\overline z}(z_0)\overline{\varphi'(0)}=0.
\]
Theorem~\ref{thm:main-disk} gives \(|F_z(0)|\leq1-|F(0)|^2\). Since \(F_z(0)=f_z(z_0)\varphi'(0)\), the estimate follows. Sharpness and the equality statement follow from
Theorem~\ref{thm:main-disk}, since precomposition by $\varphi$ preserves
the class of holomorphic automorphisms.
\end{proof}

A conformal change of variable gives the corresponding estimate on any
simply connected hyperbolic plane domain.

\begin{corollary}\label{cor:disk-hyperbolic-domain}
Let \(\Omega\subsetneq\C\) be simply connected, and let \(\lambda_\Omega(w)|dw|^2\) be its hyperbolic metric of curvature \(-1\). If \(f:\Omega\to\D\) is harmonic and \(f_{\overline w}(w_0)=0\), then
\[
 |f_w(w_0)|\leq
 \frac{\sqrt{\lambda_\Omega(w_0)}}2\bigl(1-|f(w_0)|^2\bigr).
\]
\end{corollary}

\begin{proof}
We choose a conformal map \(h:\D\to\Omega\) with \(h(0)=w_0\). Applying Theorem~\ref{thm:main-disk} to \(f\circ h\), we have
\[
 (f\circ h)_{\overline z}(0)=f_{\overline w}(w_0)\overline{h'(0)}=0,
 \qquad
 (f\circ h)_z(0)=f_w(w_0)h'(0).
\]
Since \(\lambda_\Omega(w_0)|h'(0)|^2=4\), the estimate follows.
\end{proof}

Conjugation gives the orientation-reversing counterpart.

\begin{remark}\label{rem:anti-holomorphic}
If \(f:\D\to\D\) is harmonic, \(f(0)=p\), and \(f_z(0)=0\), then
\(|f_{\overline z}(0)|\leq1-|p|^2\), with equality precisely for
anti-holomorphic disk automorphisms with prescribed value at the origin.
\end{remark}

\section{The unit ball and Cayley-Klein geometry}\label{sec:ball-CK}

For Hilbert balls, the supporting-hyperplane condition in
Proposition~\ref{prop:section-support-criterion} follows from the real
Cauchy-Schwarz inequality. The radius of the section determines the
extremal value, and the contact sets determine equality.

For \(x_0\in H\) and \(R>0\), write
\[
 \B_H(x_0,R)=\{x\in H:|x-x_0|<R\}.
\]
Thus $\B_H=\B_H(0,1)$. For $p\in\B_H(0,R)$ and a real two-plane
$\Lambda$, put $b=|p-\pi_\Lambda p|$. The section
$(p+\Lambda)\cap\B_H(0,R)$ has
radius \(c=\sqrt{R^2-b^2}\). Every point \(q\) of its boundary circle has
norm \(R\), and the real Cauchy-Schwarz inequality gives
\(\langle X,q\rangle\leq R^2\) for \(X\in\overline{\B_H(0,R)}\). Hence
Proposition~\ref{prop:section-support-criterion} gives
\begin{equation}\label{eq:directional-ball-value}
 M_{\B_H(0,R)}(p,\Lambda)
 =\frac{R^2-|p|^2}{\sqrt{R^2-b^2}}.
\end{equation}

\subsection{The unrestricted center bound}

We begin with the sharp center estimate without any conformality or
distortion assumption on the differential. It will be compared below with
the conformal bound and with the sharp bound under pointwise $K$-distortion.

The estimate follows directly from the Poisson derivative formula applied
to a real projection in a direction of maximal stretching.

\begin{proposition}\label{prop:hilbert-unrestricted-center}
Let \(H\) be a real Hilbert space and let \(F:\D\to\B_H\) be harmonic with \(F(0)=0\). Then
\begin{equation}\label{eq:hilbert-unrestricted-center}
 \|dF_0\|\leq\frac4\pi.
\end{equation}
The constant is sharp in every nonzero Hilbert space.
\end{proposition}

\begin{proof}
If \(dF_0=0\), the inequality holds. Otherwise, we choose a unit vector
$\xi_0=(\cos\theta,\sin\theta)\in\R^2$ such that
\[
 |dF_0\xi_0|=\|dF_0\|,
\]
and set $h=dF_0\xi_0/\|dF_0\|$. By Lemma~\ref{lem:hilbert-boundary},
we write $F=P[\Phi]$ with $|\Phi|\leq1$ almost everywhere. Applying
Lemma~\ref{lem:real-poisson} to $\langle F,h\rangle$ gives
\[
\begin{aligned}
 \|dF_0\|
 &=2\int_\T\langle\Phi(e^{it}),h\rangle\cos(t-\theta)\dm\\
 &\leq2\int_\T|\cos(t-\theta)|\dm=\frac4\pi.
\end{aligned}
\]
For sharpness, we take a unit vector $h\in H$, set
$U(e^{it})=\operatorname{sgn}(\cos t)$, and write $u=P[U]$.
The strictly positive Poisson kernel gives positive mass to both
semicircles on which $U$ is constant, so $-1<u(z)<1$ in $\D$.
Symmetry and Lemma~\ref{lem:real-poisson} give
\[
 u(0)=u_y(0)=0,\qquad u_x(0)=\frac4\pi.
\]
Thus $F=uh$ maps $\D$ into $\B_H$, fixes zero, and attains equality.
\end{proof}

\subsection{The conformal estimate for a prescribed value}

Specializing \eqref{eq:directional-ball-value} to the unit ball gives the
following prescribed-value estimate. Its equality characterization follows
from the fact that each supporting hyperplane meets the closed ball at a
single point.

\begin{theorem}\label{thm:ball}
Let \(H\) be a real Hilbert space with \(\dim H\geq2\), and let
\(F:\D\to\B_H\) be harmonic. Suppose that \(F(0)=p\) and that \(dF_0\neq0\) is conformal. Put
\[
 \Lambda=dF_0(\R^2),
\]
let \(\pi_\Lambda:H\to\Lambda\) denote the orthogonal projection, and let \(\lambda>0\) be the conformal factor of \(dF_0\). Then
\begin{equation}\label{eq:ball-estimate}
 \lambda\leq
 \frac{1-|p|^2}{\sqrt{1-|p|^2+|\pi_\Lambda p|^2}}.
\end{equation}
The estimate is sharp. Equality holds if and only if \(F\) is a conformal or anticonformal diffeomorphism of \(\D\) onto the affine disk
\[
 (p+\Lambda)\cap\B_H.
\]
\end{theorem}

\begin{proof}
We write
\[
 p=p_\Lambda+p_\perp,
 \qquad p_\Lambda=\pi_\Lambda p,
 \qquad p_\perp\perp\Lambda,
\]
and set $a=|p_\Lambda|$ and $b=|p_\perp|$. The affine section
$(p+\Lambda)\cap\B_H$ is the disk centered at $p_\perp$ with radius
\[
 c=\sqrt{1-b^2}.
\]
Thus \eqref{eq:directional-ball-value}, with $R=1$, gives
\[
 \lambda\leq M_{\B_H}(p,\Lambda)
 =\frac{1-|p|^2}{c}.
\]
Since
\[
 c^2=1-b^2
 =1-|p|^2+|\pi_\Lambda p|^2,
\]
this is \eqref{eq:ball-estimate}. The M\"obius parametrization of the affine
disk through $p$ attains this value, so the estimate is sharp.

For equality, we use the normalization in
\eqref{eq:directional-normalization} and write $T=E_s^c$, where $s=a/c$.
If $|y|=1$, $|X|\leq1$, and $\langle X,y\rangle=1$, then
\[
 |X-y|^2=|X|^2+1-2\langle X,y\rangle=|X|^2-1\leq0,
\]
so $X=y$. Thus every contact set in
Proposition~\ref{prop:section-support-criterion} is a singleton.
That proposition identifies the normalized extremal with $T$.
Undoing the rotation or reflection of $\D$ gives precisely the conformal
and anticonformal parametrizations of $(p+\Lambda)\cap\B_H$.
Conversely, each such parametrization is obtained from $T$ by one of
these changes of variable and has the same differential norm at the
origin, so it attains equality.
\end{proof}

The same difference formula yields a quantitative stability estimate.

\begin{theorem}\label{thm:ball-boundary-stability}
Let \(p\in\B_H\) and let \(\Lambda\subset H\) be a real two-dimensional
subspace. With the normalization from Section~\ref{sec:target-geometry},
write
\[
 p=p_\Lambda+p_\perp
\]
and set
\[
 a=|p_\Lambda|,\qquad b=|p_\perp|,\qquad
 c=\sqrt{1-b^2},\qquad s=\frac ac.
\]
Let \(T=E_s^c\) be the normalized M\"obius parametrization of
\((p+\Lambda)\cap\B_H\), and let \(\Phi_T=T|_\T\). For every
$F=P[\Phi_F]\in\mathcal F_{\B_H}^{\mathrm n}(p,\Lambda)$, set
\[
 \delta=\|dT_0\|-\|dF_0\|\geq0.
\]
Then
\begin{equation}\label{eq:ball-deficit-identity}
\begin{aligned}
 2c\delta={}&\int_\T w_s|\Phi_F-\Phi_T|^2\dm\\
 &+\int_\T w_s(1-|\Phi_F|^2)\dm.
\end{aligned}
\end{equation}
Consequently,
\begin{equation}\label{eq:ball-boundary-stability}
 \int_\T w_s|\Phi_F-\Phi_T|^2\dm\leq2c\delta,
 \qquad
 \|\Phi_F-\Phi_T\|_{L^2}^2\leq\frac{2c\delta}{(1-s)^2}.
\end{equation}
In particular, for fixed $p$ and $\Lambda$, the boundary functions of every
normalized maximizing sequence converge in $L^2$ to $\Phi_T$.
\end{theorem}

\begin{proof}
Corollary~\ref{cor:mobius-extremality} and the ball equality
$I_s^c(\Phi_T)=(1+s^2)$ give
\[
 c\delta=\int_\T w_s(1-\langle\Phi_F,\Phi_T\rangle)\dm.
\]
For $U,V\in H$ with $|V|=1$,
\[
 |U-V|^2+1-|U|^2=2(1-\langle U,V\rangle).
\]
Since $|\Phi_T|=1$ almost everywhere, applying this identity to
$U=\Phi_F$ and $V=\Phi_T$ gives \eqref{eq:ball-deficit-identity}.
Both terms on the right are nonnegative because $|\Phi_F|\leq1$ and
$w_s>0$. Discarding the second term and using
$w_s\geq(1-s)^2$ gives \eqref{eq:ball-boundary-stability}. The final
assertion follows because $\delta\to0$ along every normalized maximizing
sequence.
\end{proof}

At the center, Theorem~\ref{thm:ball} gives the sharp bound
\[
 \|dF_0\|\leq1
\]
whenever \(dF_0\neq0\) is conformal. Proposition~\ref{prop:hilbert-unrestricted-center} gives \(4/\pi\) without this hypothesis. Section~\ref{sec:center-distortion} determines the sharp constants between these two regimes for maps that are \(K\)-quasiconformal at the origin in the pointwise sense.

Precomposition with automorphisms of $\D$ gives the corresponding
pointwise estimate.

\begin{corollary}\label{cor:ball-pointwise-hilbert}
Let \(H\) be a real Hilbert space with \(\dim H\geq2\), and let
\(F:\D\to\B_H\) be harmonic. Suppose that \(dF_{z_0}\neq0\) is conformal, let
\[
 p=F(z_0),\qquad \Lambda=dF_{z_0}(\R^2),
\]
and let \(\lambda_{z_0}\) be the conformal factor. Then
\begin{equation}\label{eq:ball-pointwise}
 \lambda_{z_0}\leq
 \frac{1-|p|^2}
 {(1-|z_0|^2)\sqrt{1-|p|^2+|\pi_\Lambda p|^2}}.
\end{equation}
Equality holds if and only if \(F\) is a conformal or anticonformal diffeomorphism of \(\D\) onto \((p+\Lambda)\cap\B_H\).
\end{corollary}

\begin{proof}
We choose \(\psi\in\Aut(\D)\) with \(\psi(0)=z_0\). The map \(F\circ\psi\) is harmonic, its differential at the origin is conformal with factor \((1-|z_0|^2)\lambda_{z_0}\), and it has the same image plane \(\Lambda\). Applying Theorem~\ref{thm:ball} gives the estimate, and the equality statement is preserved under precomposition by \(\psi\).
\end{proof}

In finite dimensions, writing $\B^n=\B_{\R^n}$,
Corollary~\ref{cor:ball-pointwise-hilbert} recovers the estimate and
equality cases of \cite{ForstnericKalaj}. We record the resulting formula
in terms of the angle between $F(z_0)$ and the image plane of the
differential.

\begin{corollary}\label{cor:finite-dimensional-ball}
Let \(n\geq2\), let \(F:\D\to\B^n\) be harmonic, and suppose that
\(dF_{z_0}\neq0\) is conformal. If \(F(z_0)\neq0\), let \(\theta\in[0,\pi/2]\) be the angle between \(F(z_0)\) and \(\Lambda=dF_{z_0}(\R^2)\); when \(F(z_0)=0\), let \(\theta\) be arbitrary. Then
\[
 \|dF_{z_0}\|\leq
 \frac{1-|F(z_0)|^2}
 {(1-|z_0|^2)\sqrt{1-|F(z_0)|^2\sin^2\theta}}.
\]
Equality holds if and only if \(F\) is a conformal or anticonformal diffeomorphism of \(\D\) onto the affine disk \((F(z_0)+\Lambda)\cap\B^n\).
\end{corollary}

\begin{proof}
Since \(|\pi_\Lambda F(z_0)|=|F(z_0)|\cos\theta\), the denominator in \eqref{eq:ball-pointwise} is
\[
 \sqrt{1-|F(z_0)|^2+|\pi_\Lambda F(z_0)|^2}
 =\sqrt{1-|F(z_0)|^2\sin^2\theta}.
\]
The result therefore follows from Corollary~\ref{cor:ball-pointwise-hilbert}.
\end{proof}

Writing $p=F(z_0)$, in dimension two one has $\Lambda=\R^2$, and the
estimate becomes
\[
 \|dF_{z_0}\|\leq\frac{1-|p|^2}{1-|z_0|^2}.
\]
If \(n\geq3\) and \(\Lambda\perp p\), it becomes
\[
 \|dF_{z_0}\|\leq\frac{\sqrt{1-|p|^2}}{1-|z_0|^2}.
\]
The quantity
\[
 c=\sqrt{1-|p|^2+|\pi_\Lambda p|^2}
\]
is exactly the radius of the affine disk \((p+\Lambda)\cap\B^n\). The
following example first realizes equality in an \(L^2\) space and then
gives a separate harmonic map whose range is not contained in any
finite-dimensional subspace.

\begin{example}\label{ex:L2-ball}
Let
\[
 H=L^2(\T,m;\R),
\]
and define
\[
 e_0(e^{it})=1,\qquad
 e_1(e^{it})=\sqrt2\cos t,\qquad
 e_2(e^{it})=\sqrt2\sin t.
\]
These functions form an orthonormal triple in \(H\). We choose \(a\geq0\) and \(b_0\in\R\) with \(a^2+b_0^2<1\), and set
\[
 c=\sqrt{1-b_0^2},\qquad s=\frac ac.
\]
Then \(0\leq s<1\). Using the M\"obius automorphism \(\mathcal M_s\) from Section~\ref{sec:weighted-support}, define
\begin{equation}\label{eq:L2-extremal}
 F(z)=b_0 e_0+c\bigl(\re\mathcal M_s(z)\,e_1+\operatorname{Im}\mathcal M_s(z)\,e_2\bigr).
\end{equation}
The map \(F:\D\to H\) is harmonic. Since \(e_0,e_1,e_2\) are orthonormal and \(|\mathcal M_s(z)|<1\),
\[
 |F(z)|^2=b_0^2+c^2|\mathcal M_s(z)|^2< b_0^2+c^2=1,
\]
so \(F(\D)\subset\B_H\). Moreover,
\[
 F(0)=b_0 e_0+a e_1,
\]
and, because \(\mathcal M_s'(0)=1-s^2\),
\[
 F_x(0)=c(1-s^2)e_1,\qquad
 F_y(0)=c(1-s^2)e_2.
\]
Thus \(dF_0\) is conformal with image plane \(\Lambda=\operatorname{span}\{e_1,e_2\}\) and conformal factor
\[
 c(1-s^2)=\frac{1-a^2-b_0^2}{\sqrt{1-b_0^2}}.
\]
Since \(\pi_\Lambda F(0)=a e_1\), the right-hand side of
\eqref{eq:ball-estimate} has the same value. Hence equality holds in
\eqref{eq:ball-estimate}. Because \(\mathcal M_s\) is an automorphism of
\(\D\), the map \(F\) is a conformal diffeomorphism of \(\D\) onto the
affine disk
\[
 \bigl(F(0)+\Lambda\bigr)\cap\B_H
 =\bigl(b_0 e_0+\Lambda\bigr)\cap\B_H.
\]

The class covered by Theorem~\ref{thm:ball} is not restricted to maps with finite-dimensional range. We choose an orthonormal sequence \((u_n)_{n\geq2}\) in \(\operatorname{span}\{e_1,e_2\}^\perp\), a number \(\lambda_\infty>0\), and positive numbers \(\gamma_n\) such that
\[
 \lambda_\infty^2+\sum_{n=2}^\infty\gamma_n^2<1.
\]
The series
\begin{equation}\label{eq:L2-infinite-range}
 F_\infty(z)=\lambda_\infty\bigl(\re z\,e_1+\operatorname{Im}z\,e_2\bigr)
 +\sum_{n=2}^\infty\gamma_n\re(z^n)u_n
\end{equation}
converges together with all derivatives uniformly on compact subsets of \(\D\), and hence defines an \(H\)-valued harmonic map. Indeed, fix an integer \(q\geq0\) and \(0<r<1\). The norm of any derivative of order \(q\) of the \(n\)-th summand on \(|z|\leq r\) is at most \(\gamma_n n^q r^{n-q}\) when \(n\geq q\), and it vanishes when \(n<q\). Hence, by Cauchy-Schwarz,
\[
 \sum_{n=\max\{2,q\}}^\infty \gamma_n n^q r^{n-q}
 \leq
 \left(\sum_{n=2}^\infty\gamma_n^2\right)^{1/2}
 \left(\sum_{n=\max\{2,q\}}^\infty n^{2q}r^{2n-2q}\right)^{1/2}
 <\infty.
\]
Thus every differentiated series converges uniformly on compact subsets of \(\D\). Orthogonality gives
\[
 \begin{aligned}
 |F_\infty(z)|^2
 &=\lambda_\infty^2|z|^2+\sum_{n=2}^\infty\gamma_n^2\bigl(\re(z^n)\bigr)^2\\
 &\leq\lambda_\infty^2|z|^2+\sum_{n=2}^\infty\gamma_n^2|z|^{2n}<1.
 \end{aligned}
\]
Moreover, \((F_\infty)_x(0)=\lambda_\infty e_1\) and
\((F_\infty)_y(0)=\lambda_\infty e_2\), so \(d(F_\infty)_0\) is conformal with
factor \(\lambda_\infty\). The range of \(F_\infty\) is not contained in any
finite-dimensional subspace. Indeed, if \(F_\infty(\D)\subset V\) for a
finite-dimensional subspace \(V\subset H\), then its restriction to the real
interval would give
\[
 \left.\frac{d^n}{dr^n}F_\infty(r)\right|_{r=0}
 =n!\gamma_n u_n\in V,
 \qquad n\geq2,
\]
which is impossible. Thus the Hilbert-ball theorem applies to harmonic maps with infinite-dimensional range. Its equality statement explains why every mapping attaining equality nevertheless has its range in an affine two-dimensional disk.
\end{example}

\subsection{A dimension-free infinitesimal Klein norm}

For a real Hilbert space \(H\), \(p\in\B_H\), and \(v\in H\), we define
\begin{equation}\label{eq:CK-def}
 \CK_{\B_H}(p;v)=
 \frac{\sqrt{(1-|p|^2)|v|^2+\langle p,v\rangle^2}}
 {1-|p|^2}.
\end{equation}
In finite dimensions, \eqref{eq:CK-def} is the infinitesimal norm of the
Beltrami-Klein model of real hyperbolic space. We use the same formula in an
arbitrary real Hilbert space. It is not the conformal Poincar\'e norm; for
the disk we write
\begin{equation}\label{eq:PD-def}
 P_\D(z;v)=\frac{|v|}{1-|z|^2}
\end{equation}
for the hyperbolic norm in the Schwarz-Pick normalization, and
\begin{equation}\label{eq:poincare-curv-minus-one}
 ds_P(z;v)=\frac{2|v|}{1-|z|^2}=2P_\D(z;v)
\end{equation}
for the Poincar\'e metric of curvature \(-1\).

We first determine the largest value of the Klein norm on a fixed subspace.

\begin{lemma}\label{lem:CK-plane}
Let \(H\) be a real Hilbert space, let \(p\in\B_H\), and let \(\Lambda\subset H\) be a nonzero closed linear subspace. Then, for every \(v\in\Lambda\),
\begin{equation}\label{eq:CK-plane-bound}
 \CK_{\B_H}(p;v)\leq
 \frac{\sqrt{1-|p|^2+|\pi_\Lambda p|^2}}
 {1-|p|^2}|v|.
\end{equation}
The constant is the largest value of \(\CK_{\B_H}(p;v)\) over Euclidean unit vectors \(v\in\Lambda\). If \(\pi_\Lambda p\neq0\), equality for a nonzero vector \(v\in\Lambda\) holds precisely when \(v\) is parallel to \(\pi_\Lambda p\); if \(\pi_\Lambda p=0\), equality holds for every \(v\in\Lambda\).
\end{lemma}

\begin{proof}
For \(v\in\Lambda\),
\[
 \langle p,v\rangle=\langle\pi_\Lambda p,v\rangle.
\]
Cauchy-Schwarz gives \(|\langle p,v\rangle|\leq|\pi_\Lambda p|\,|v|\), and substitution in \eqref{eq:CK-def} proves the estimate and its equality statement.
\end{proof}

In the applications below, \(\Lambda=dF_{z_0}(\R^2)\) is two-dimensional
under either the conformality hypothesis or the pointwise distortion
condition. The lemma is stated for an arbitrary nonzero closed subspace
because its proof does not use the dimension of \(\Lambda\).

If $\dim H\geq2$, the norm \eqref{eq:CK-def} is anisotropic at every
$p\neq0$. On Euclidean unit vectors orthogonal to $p$ it equals
$1/\sqrt{1-|p|^2}$, while in the radial direction it equals
$1/(1-|p|^2)$. Lemma~\ref{lem:CK-plane} records its largest value on the
image plane of the differential.

\subsection{A Cayley-Klein contraction under pointwise distortion}

We first prove the contraction under the pointwise $K$-distortion
condition. The proof uses the M\"obius-weighted pairing; the rotation
parameter is optimized before the two complex derivatives are compared.

\begin{theorem}
\label{thm:qc-CK-infinitesimal}
Let $H$ be a real Hilbert space, let $K\geq1$, let
$F:\D\to\B_H$ be harmonic, and let $z_0\in\D$. If $F$ is
$K$-quasiconformal at $z_0$ in the pointwise sense, then, for every
$v\in T_{z_0}\D$,
\begin{equation}\label{eq:qc-CK-infinitesimal}
 \CK_{\B_H}\bigl(F(z_0);dF_{z_0}v\bigr)
 \leq K\,P_\D(z_0;v).
\end{equation}
More precisely, put $p=F(z_0)$, $\Lambda=dF_{z_0}(\R^2)$, and
\[
 c=\sqrt{1-|p|^2+|\pi_\Lambda p|^2},\qquad
 s=\frac{|\pi_\Lambda p|}{c}.
\]
Then $0\leq s<1$ and
\begin{equation}\label{eq:qc-CK-refined}
 \CK_{\B_H}(p;dF_{z_0}v)
 \leq\frac{2K}{K+1-(K-1)s^2}\,P_\D(z_0;v).
\end{equation}
\end{theorem}

\begin{proof}
We first prove the estimate at the origin. We put $p=F(0)$ and
$\Lambda=dF_0(\R^2)$. Since $F$ is $K$-quasiconformal at $0$, the plane
$\Lambda$ is two-dimensional. We write
\[
 p=p_\Lambda+p_\perp,
 \qquad p_\Lambda=ae_1,
 \qquad a=|p_\Lambda|,
 \qquad b=|p_\perp|,
\]
where $(e_1,e_2)$ is an orthonormal basis of $\Lambda$, with $e_1$ chosen
arbitrarily when $a=0$. The affine section $(p+\Lambda)\cap\B_H$ is centered
at $p_\perp$ and has radius $c=\sqrt{1-b^2}$. We set $s=a/c$. Since
$p\in\B_H$, one has $0\leq s<1$.

We denote by $f_\Lambda$ the complex coordinate of the $\Lambda$-component of
$F$ and write
\[
 d(f_\Lambda)_0(z)=Az+B\overline z.
\]
Precomposing with the reflection $z\mapsto\overline z$, if necessary,
interchanges $A$ and $B$ without changing $p$, $\Lambda$, or the two
stretchings. We may therefore assume that $|A|\geq|B|$. Since $dF_0$ takes
values in $\Lambda$, Lemma~\ref{lem:operator-norm} gives
\begin{equation}\label{eq:qc-AB-stretchings}
 L_F(0)=|A|+|B|,
 \qquad
 \ell_F(0)=|A|-|B|.
\end{equation}

For $\theta\in\R$, define $F_\theta(z)=F(e^{i\theta}z)$. The two
coefficients of the differential of its $\Lambda$-component are
$e^{i\theta}A$ and $e^{-i\theta}B$. We denote by $\Phi_\theta$ the Poisson
boundary function of $F_\theta$. Since $d(F_\theta)_0(\R^2)=\Lambda$, the
normal component has mean $p_\perp$, and its first derivatives at the origin
vanish. Moreover,
$|\Phi_\theta|\leq1$ almost everywhere and $|E_s^c|=1$ on $\T$. Hence the
real Cauchy-Schwarz inequality gives
\[
 I_s^c(\Phi_\theta)\leq1+s^2.
\]
Theorem~\ref{thm:weighted-pairing-identity} now yields
\[
 c\re\bigl(e^{i\theta}A+s^2e^{-i\theta}B\bigr)
 +2csa+(1+s^2)b^2\leq1+s^2.
\]
The identities $cs=a$, $c^2=1-b^2$, and $|p|^2=a^2+b^2$ reduce this to
\begin{equation}\label{eq:qc-rotated-pairing}
 c\re\bigl(e^{i\theta}A+s^2e^{-i\theta}B\bigr)
 \leq1-|p|^2.
\end{equation}
Since
\[
 \re\bigl(e^{i\theta}A+s^2e^{-i\theta}B\bigr)
 =\re\bigl(e^{i\theta}(A+s^2\overline B)\bigr),
\]
taking the maximum over $\theta$ in \eqref{eq:qc-rotated-pairing} gives
\[
 c|A+s^2\overline B|\leq1-|p|^2.
\]
We write $L=L_F(0)$ and $\ell=\ell_F(0)$. Then
\[
\begin{aligned}
 |A+s^2\overline B|
 &\geq |A|-s^2|B|\\
 &=\frac{(1-s^2)L+(1+s^2)\ell}{2}\\
 &\geq\frac{K+1-(K-1)s^2}{2K}\,L.
\end{aligned}
\]
Consequently,
\begin{equation}\label{eq:qc-euclidean-ball-bound}
 L_F(0)\leq
 \frac{2K(1-|p|^2)}{c\bigl(K+1-(K-1)s^2\bigr)}.
\end{equation}
No simultaneous phase normalization of $A$ and $B$ has been used; the
direction $e_1$ is fixed by $p_\Lambda$.

For $y\in\Lambda$, Lemma~\ref{lem:CK-plane} and the identity
$c^2=1-|p|^2+|\pi_\Lambda p|^2$ give
\[
 \CK_{\B_H}(p;y)\leq\frac{c}{1-|p|^2}|y|.
\]
Applying this estimate to $y=dF_0v$ and using
\eqref{eq:qc-euclidean-ball-bound}, we conclude that
\[
 \CK_{\B_H}(p;dF_0v)
 \leq\frac{c}{1-|p|^2}L_F(0)|v|
 \leq\frac{2K}{K+1-(K-1)s^2}|v|.
\]

For a general $z_0\in\D$, we choose $\psi\in\Aut(\D)$ with
$\psi(0)=z_0$ and apply the estimate at the origin to $F\circ\psi$.
Conformal precomposition preserves the pointwise distortion ratio, and
$|\psi'(0)|=1-|z_0|^2$. For $v\in T_{z_0}\D$, we set
$w=(d\psi_0)^{-1}v$. Then $|w|=|v|/(1-|z_0|^2)$, and the estimate for
$F\circ\psi$ gives
\[
 \CK_{\B_H}\bigl(F(z_0);dF_{z_0}v\bigr)
 \leq\frac{2K}{K+1-(K-1)s^2}\frac{|v|}{1-|z_0|^2}.
\]
This proves \eqref{eq:qc-CK-refined}. Since
$K+1-(K-1)s^2\geq2$, it implies
\eqref{eq:qc-CK-infinitesimal}. For $K>1$ the refined factor is
strictly smaller than $K$, because $s<1$.
\end{proof}

For $p,q\in\B_H$, define the path distance induced by \eqref{eq:CK-def} by
\begin{equation}\label{eq:CK-path-distance}
 \CK_{\B_H}(p,q)
 =\inf_\gamma\int_0^1
 \CK_{\B_H}\bigl(\gamma(t);\gamma'(t)\bigr)\,dt,
\end{equation}
where the infimum is taken over all piecewise $C^1$ curves
$\gamma:[0,1]\to\B_H$ joining $p$ to $q$. We use the normalization
\begin{equation}\label{eq:HarD-def}
 d_\D(z_0,z)
 =\inf_\gamma\int_0^1
 P_\D\bigl(\gamma(t);\gamma'(t)\bigr)\,dt
\end{equation}
for the distance induced by $P_\D$, with the infimum taken over piecewise $C^1$
curves joining $z_0$ to $z$. Equivalently,
\[
 d_\D(z_0,z)
 =\operatorname{arctanh}
 \left|\frac{z-z_0}{1-\overline{z_0}z}\right|.
\]

Integrating the pointwise estimate along curves gives the corresponding
global Schwarz-Pick estimate.

\begin{theorem}
\label{thm:qc-CK-global}
Let $H$ be a real Hilbert space, let $K\geq1$, and let
$F:\D\to\B_H$ be harmonic. Suppose that $F$ is $K$-quasiconformal at every
point of $\D$ in the pointwise sense. Then, for all $z_0,z\in\D$,
\begin{equation}\label{eq:qc-CK-global}
 \CK_{\B_H}\bigl(F(z_0),F(z)\bigr)
 \leq K\,d_\D(z_0,z).
\end{equation}
\end{theorem}

\begin{proof}
We consider a piecewise $C^1$ curve $\gamma:[0,1]\to\D$ joining $z_0$ to $z$.
The definition of the path distance and
Theorem~\ref{thm:qc-CK-infinitesimal} give
\[
\begin{aligned}
 \CK_{\B_H}\bigl(F(z_0),F(z)\bigr)
 &\leq \int_0^1
 \CK_{\B_H}\bigl(F(\gamma(t));(F\circ\gamma)'(t)\bigr)\,dt
 \\
 &\leq K\int_0^1
 P_\D\bigl(\gamma(t);\gamma'(t)\bigr)\,dt.
\end{aligned}
\]
After taking the infimum over all such curves, we obtain \eqref{eq:qc-CK-global}.
\end{proof}

\subsection{The conformal case}

Taking $K=1$ in Theorem~\ref{thm:qc-CK-infinitesimal} gives the conformal
contraction. The prescribed-value estimate in
Corollary~\ref{cor:ball-pointwise-hilbert} also determines when equality can
occur.

\begin{corollary}\label{cor:ball-CK-contraction}
Let \(H\) be a real Hilbert space, let \(F:\D\to\B_H\) be harmonic, and
suppose that \(dF_{z_0}\neq0\) is conformal. Then, for every
\(v\in T_{z_0}\D\),
\begin{equation}\label{eq:ball-CK-contraction}
 \CK_{\B_H}\bigl(F(z_0);dF_{z_0}v\bigr)
 \leq P_\D(z_0;v).
\end{equation}
Equivalently,
\[
 \CK_{\B_H}\bigl(F(z_0);dF_{z_0}v\bigr)
 \leq\frac12 ds_P(z_0;v).
\]
If \(v\neq0\), equality holds if and only if \(F\) is a conformal or
anticonformal diffeomorphism of \(\D\) onto the affine disk
\((F(z_0)+\Lambda)\cap\B_H\), where
\(\Lambda=dF_{z_0}(\R^2)\), and, when
\(\pi_\Lambda F(z_0)\neq0\), the vector \(dF_{z_0}v\) is parallel to
\(\pi_\Lambda F(z_0)\). When the projection vanishes, the directional
condition is vacuous, and every parametrization described above gives equality
in every nonzero tangent direction.
\end{corollary}

\begin{proof}
A nonzero conformal differential has equal maximal and minimal stretchings,
so \(F\) is \(1\)-quasiconformal at \(z_0\) in the pointwise sense.
The result is therefore Theorem~\ref{thm:qc-CK-infinitesimal} with \(K=1\).

For the equality statement, Lemma~\ref{lem:CK-plane} and
\eqref{eq:ball-pointwise} give \eqref{eq:ball-CK-contraction} directly.
Equality requires equality in both estimates. The first gives the stated
directional condition, and the equality statement in
Corollary~\ref{cor:ball-pointwise-hilbert} gives the affine-disk
parametrization. The converse follows from the same two results. For
\(H=\R^n\), this agrees with the finite-dimensional equality geometry.
\end{proof}

\subsection{The geometric meaning of the projection factor}
\label{subsec:klein-factor-records}

The factor $c/(1-|p|^2)$ in Lemma~\ref{lem:CK-plane} is the reciprocal
of the sharp conformal factor $(1-|p|^2)/c$. Their product gives the
constant $1$ in Corollary~\ref{cor:ball-CK-contraction}.

The vector \(p\) and the linear two-plane \(\Lambda\) play different
roles; in particular, \(p\in\Lambda\) is not assumed. The equality
\[
 \langle p,y\rangle=\langle\pi_\Lambda p,y\rangle,
 \qquad y\in\Lambda,
\]
shows why only the projection is visible to the differential. If \(p\neq0\)
and \(\pi_\Lambda p=0\), then the affine plane \(p+\Lambda\) is contained in
the tangent hyperplane at \(p\) to the sphere centered at the origin and
passing through \(p\), and the Klein norm is constant on Euclidean unit
vectors in \(\Lambda\). When \(p=0\), the Klein norm is isotropic on every two-dimensional
subspace. If \(p\in\Lambda\) and \(p\neq0\), the plane contains the radial
direction, where the Klein norm is maximal. In the remaining cases, the
maximal value on \(\Lambda\) is determined by \(|\pi_\Lambda p|\), as in
Lemma~\ref{lem:CK-plane}.

The factor $K$ in Theorem~\ref{thm:qc-CK-infinitesimal} is uniform.
The refinement \eqref{eq:qc-CK-refined} retains the dependence on $s$
and gives $L_F(0)\leq2K/(K+1)$ at the center. For $K>1$, the optimal constant $M_K$ obtained in
Section~\ref{sec:center-distortion} is smaller still. An additional alignment hypothesis gives the estimate in
Proposition~\ref{prop:qc-pairing-estimate}.

\section{Cylinder targets and aligned differential distortion}\label{sec:cylinder-aligned}

\subsection{Horizontal and nonhorizontal directions}

Cylindrical targets show why boundedness matters for the
supporting-hyperplane estimates used above. For a real Hilbert space
\(H_1\), we consider the cylinder
\[
 \mathcal C=\D\times H_1\subset \C\oplus H_1.
\]
If $H_1\neq\{0\}$, the cylinder is unbounded in the vertical directions,
and every real linear functional with a nonzero $H_1$-component is
unbounded on the target. A finite optimal bound remains available when the
differential is horizontal.

\begin{proposition}\label{prop:horizontal-cylinder}
Let \(F=(f,V):\D\to\mathcal C\) be harmonic. Suppose that \(dF_0\) is either zero or conformal, and is horizontal, that is,
\[
 dF_0(\R^2)\subset \C\oplus\{0\}.
\]
Then
\begin{equation}\label{eq:horizontal-cylinder-bound}
 \|dF_0\|\leq 1-|f(0)|^2.
\end{equation}
The estimate is sharp. Equality holds if and only if $f$ is a conformal
or anticonformal automorphism of $\D$. The vertical component may be any
harmonic map with $dV_0=0$.
\end{proposition}

\begin{proof}
If $dF_0=0$, the inequality is strict because $|f(0)|<1$.
Otherwise, horizontality gives $dV_0=0$ and
$\|dF_0\|=\|df_0\|$. The planar differential $df_0$ is conformal or
anticonformal. Theorem~\ref{thm:main-disk} and
Remark~\ref{rem:anti-holomorphic} give
\eqref{eq:horizontal-cylinder-bound} and show that equality is equivalent
to $f$ being a conformal or anticonformal disk automorphism. Conversely,
any such $f$, together with a harmonic $V$ satisfying $dV_0=0$, attains
equality.
\end{proof}

The sharpness statement is independent of the choice of vertical Hilbert space. Let \(H_1\neq\{0\}\), choose a unit vector \(h\in H_1\), fix \(0\leq r<1\), and put
\[
 A_r(z)=\frac{r+z}{1+rz},
 \qquad
 G_\kappa(z)=\kappa\,\re(z^2)h,
 \qquad \kappa>0.
\]
Then
\begin{equation}\label{eq:cylinder-example}
 F_\kappa(z)=\bigl(A_r(z),G_\kappa(z)\bigr)
\end{equation}
maps \(\D\) harmonically into \(\D\times H_1\), satisfies \(F_\kappa(0)=(r,0)\), and has horizontal tangent plane at the origin. Moreover,
\[
 \|d(F_\kappa)_0\|=|A_r'(0)|=1-r^2.
\]
On the boundary,
\[
 \|F_\kappa(e^{it})\|_{\C\oplus H_1}^2
 =|A_r(e^{it})|^2+\kappa^2\cos^2(2t)
 =1+\kappa^2\cos^2(2t).
\]
Thus \(\|d(F_\kappa)_0\|=1-r^2\) is attained although the boundary values do not lie in the Hilbert unit ball when \(\kappa>0\).

Assume first that \(\dim H_1\geq2\). Then the cylinder alone gives no uniform bound for the full class of nonhorizontal conformal differentials. Indeed, for orthonormal vectors \(e_1,e_2\in H_1\) and any \(L>0\), the map
\[
 F_L(x+iy)=\bigl(0,Lxe_1+Lye_2\bigr)
\]
is harmonic from \(\D\) into \(\mathcal C\), is conformal at every point, and satisfies \(\|d(F_L)_0\|=L\). Thus, when the vertical factor has dimension at least two, no finite estimate for \(\|dF_0\|\) can follow from the cylindrical target alone.

The case of a one-dimensional vertical factor is different and admits an
optimal universal bound.

\begin{remark}\label{rem:cylinder-one-dimensional}
The restriction on the vertical dimension is essential. Suppose that \(H_1=\R\), let \(F=(f,V):\D\to\D\times\R\) be harmonic, and assume that \(dF_0\) is conformal with factor \(\lambda\). Then \(\|dF_0\|=\lambda\). Since the orthogonal projection onto the disk factor is contractive,
\[
 \|df_0\|\leq\lambda.
\]
On the other hand, the linear functional \(dV_0:\R^2\to\R\) has a unit vector \(v\) in its kernel. Hence
\[
 \lambda=|dF_0v|=|df_0v|\leq\|df_0\|,
\]
and therefore \(\lambda=\|df_0\|\).

We write $f=P[\Phi_f]$ with $|\Phi_f|\leq1$ almost everywhere.
For every unit vector $w=(\cos\theta,\sin\theta)$, the Poisson derivative
formula gives
\[
 |df_0w|
 =\left|2\int_\T\Phi_f(e^{it})\cos(t-\theta)\dm\right|
 \leq2\int_\T|\cos(t-\theta)|\dm=\frac4\pi.
\]
Taking the supremum over \(w\) yields
\[
 \lambda=\|df_0\|\leq\frac4\pi.
\]
The constant \(4/\pi\) is sharp. We take
$u=P[U]$, where $U(e^{it})=\operatorname{sgn}(\cos t)$, as in the proof of
Proposition~\ref{prop:hilbert-unrestricted-center}. Then
$\nabla u(0)=(4/\pi,0)$. We set
\[
 F(x+iy)=\left(u(x+iy),\frac4\pi y\right).
\]
Then \(F:\D\to\D\times\R\) is harmonic and \(dF_0\) is conformal with factor \(4/\pi\).
\end{remark}

For an integrable boundary function in a nonhorizontal direction,
Corollary~\ref{cor:pairing-estimate-general} still converts any finite
estimate $I_s^c(\Phi)\leq S$ into a derivative bound. The unbounded
cylinder itself need not provide such an estimate.

\subsection{An aligned pointwise distortion estimate away from the center}

The M\"obius-weighted Hilbert pairing also yields an estimate for an
aligned nonconformal differential. We first recall the principal directions
of a real linear operator $\mathcal L:\R^2\to H$ with two-dimensional image.
The self-adjoint operator $\mathcal L^*\mathcal L$ has an orthonormal
eigenbasis $v_1,v_2$.
These are the principal stretching directions: if
$|\mathcal Lv_1|\geq|\mathcal Lv_2|$, these two quantities are the maximal
and minimal stretchings of $\mathcal L$, and their image directions are
orthogonal.

Let \(F=P[\Phi]\), decompose \(\Phi=(\xi,\eta)\) according to
\(H=\Lambda\oplus\Lambda^\perp\), and write
\[
 F(0)=p=p_\Lambda+p_\perp,
 \qquad p_\Lambda\in\Lambda,
 \qquad p_\perp\in\Lambda^\perp,
\]
and put \(b=|p_\perp|\). Suppose that \(dF_0(\R^2)=\Lambda\). We choose coordinates in
\(\R^2\) along the principal stretching directions of \(dF_0\), with the
maximal stretching first, and choose an oriented orthonormal basis
\(e_1,e_2\) of \(\Lambda\) along their images. The additional alignment
assumption is that \(p_\Lambda\) is parallel to the target direction of
maximal stretching; when \(p_\Lambda=0\), this condition is vacuous.
Choosing the signs of the principal direction corresponding to the
maximal stretching and of its image vector consistently, we may arrange
that
\[
 p_\Lambda=ae_1,
 \qquad a=|p_\Lambda|,
\]
and
\begin{equation}\label{eq:qc-normalization}
 F_x(0)=\alpha e_1,
 \qquad
 F_y(0)=\beta e_2,
 \qquad
 \alpha\geq\beta>0.
\end{equation}
Here \(\alpha=L_F(0)\) and \(\beta=\ell_F(0)\). For a conformal differential this alignment can always be achieved after fixing the direction of \(p_\Lambda\); for a general nonconformal differential it is an additional geometric hypothesis. The center case, treated below, has no alignment restriction.

For the weighted estimate, we fix \(c>0\) and \(0\leq s<1\) and use this
choice of \(p_\perp\) and \(e_1,e_2\) in the definitions of \(E_s^c\),
\(w_s\), and \(I_s^c\) from Section~\ref{sec:weighted-support}. Here
\(c\) and \(s\) are free parameters. In the round-section application below,
\(c\) is the radius of the section and \(s=a/c\), exactly as in
\eqref{eq:directional-s}.

Let \(f_\Lambda=P[\phi]\) be the complex coordinate of the \(\Lambda\)-component. Then
\[
 (f_\Lambda)_x(0)=\alpha,
 \qquad
 (f_\Lambda)_y(0)=i\beta,
\]
so
\begin{equation}\label{eq:qc-fourier-coefficients}
 \widehat\phi(1)=\frac{\alpha+\beta}{2},
 \qquad
 \widehat\phi(-1)=\frac{\alpha-\beta}{2}.
\end{equation}
Substituting these coefficients into the weighted identity gives the
following estimate.

\begin{proposition}\label{prop:qc-pairing-estimate}
Let \(H\) be a real Hilbert space, let \(\Lambda\subset H\) be an oriented
real two-plane with orthonormal basis \(e_1,e_2\), and fix \(c>0\) and
\(0\leq s<1\). Let
\(\Phi=(\xi,\eta)\in L^1(\T,\Lambda\oplus\Lambda^\perp)\), and denote by
\(\phi\) the complex coordinate of \(\xi\) under the identification
\(e_1\leftrightarrow1\), \(e_2\leftrightarrow i\). Let \(F=P[\Phi]\), and
suppose that
\[
 F(0)=ae_1+p_\perp,
 \qquad a\geq0,
 \qquad p_\perp\in\Lambda^\perp,
 \qquad b=|p_\perp|.
\]
Assume that \eqref{eq:qc-normalization} holds. If, for some \(S\in\R\), the
weighted-pairing estimate \eqref{eq:general-pairing-estimate} holds, then
\begin{equation}\label{eq:qc-basic-ineq}
 \frac c2\bigl((1+s^2)\alpha+(1-s^2)\beta\bigr)
 +2csa+(1+s^2)b^2
 \leq S.
\end{equation}
If \(\alpha/\beta\leq K\), then
\begin{equation}\label{eq:qc-general-alpha-bound}
 \alpha\leq
 \frac{2K\bigl(S-(1+s^2)b^2-2csa\bigr)}
 {c\bigl(K(1+s^2)+(1-s^2)\bigr)}.
\end{equation}
\end{proposition}

\begin{proof}
The identity \(F(0)=ae_1+p_\perp\) gives \(\widehat\phi(0)=a\), while
\eqref{eq:qc-normalization} gives \eqref{eq:qc-fourier-coefficients}.
The normal component has mean $p_\perp$ and vanishing first derivatives,
so \eqref{eq:general-orthogonal-identities} holds. We may therefore apply
Corollary~\ref{cor:pairing-estimate-general}. Substitution in
\eqref{eq:pairing-estimate-identity} gives \eqref{eq:qc-basic-ineq}. If
\(\alpha/\beta\leq K\), then \(\beta\geq\alpha/K\). Since
\(1-s^2\geq0\), the left-hand side of \eqref{eq:qc-basic-ineq} is at least
\[
 \frac c2\alpha\left((1+s^2)+\frac{1-s^2}{K}\right)
 +2csa+(1+s^2)b^2.
\]
Solving the resulting inequality for \(\alpha\) gives \eqref{eq:qc-general-alpha-bound}.
\end{proof}

For the unit ball, where \(S=1+s^2\), \(c^2+b^2=1\), and \(s=a/c\), the estimate becomes
\begin{equation}\label{eq:qc-ball-alpha-bound}
 \alpha\leq
 \frac{2K(1-|p|^2)}
 {c\bigl((K+1)+(K-1)s^2\bigr)},
 \qquad
 c=\sqrt{1-|p|^2+|\pi_\Lambda p|^2}.
\end{equation}
For $s>0$ and $K>1$, the plus sign in this denominator gives a stronger
bound than the minus sign in \eqref{eq:qc-euclidean-ball-bound}; it comes
from the alignment assumption. For \(K=1\), estimate \eqref{eq:qc-ball-alpha-bound} reduces exactly to the conformal estimate \eqref{eq:ball-estimate}. If \(K>1\), equality in \eqref{eq:qc-ball-alpha-bound} is impossible under the hypotheses used in its derivation. Indeed, equality in the real Cauchy-Schwarz inequality used for the sphere would force
\[
 \Phi(z)=E_s^c(z)
\]
for almost every \(z\in\T\). The corresponding Poisson integral parametrizes the affine disk and has a conformal differential, so
\[
 \alpha=\beta=c(1-s^2)>0.
\]
Equality in the step \(\beta\geq\alpha/K\), however, would require \(\beta=\alpha/K\). These conditions are incompatible when \(K>1\). Thus the displayed upper bound is not attained for $K>1$. This nonattainment
does not by itself determine the supremum for the nonconformal
prescribed-value problem away from the center.

Away from the center, the boundary function used above is not
simultaneously adapted to the prescribed value and to the principal
stretching directions of a general nonconformal differential. At the
center, rotational symmetry removes this incompatibility, and the
pointwise distortion problem can be solved completely.

\section{The center problem under a pointwise distortion bound}
\label{sec:center-distortion}

At the center, there is no preferred target direction, so the differential
can be normalized along its principal stretching directions. After the change of parameter \(\tau=\tan\theta\), the functions
\(\alpha(\tau)\) and \(\beta(\tau)\) below correspond to the
principal-derivative parametrization used in \cite{Zwonek2022}. We impose the additional constraint
$L_F(0)/\ell_F(0)\leq K$, determine the optimal parameter, prove its
uniqueness, and characterize all equality cases in an arbitrary real
Hilbert ball.

The value of this extremal problem is independent of the Hilbert-space
dimension once it is at least two. Indeed, projection onto
$dF_0(\R^2)$ preserves $F(0)=0$ and the differential and takes values
in the unit disk of that plane. Conversely, every planar map can be
embedded isometrically in $H$. The proof below also shows that equality
forces all orthogonal components to vanish.

Recall from Definition~\ref{def:pointwise-distortion} that \(F\) is \(K\)-quasiconformal at the origin in the pointwise sense precisely when
\[
 0<\ell_F(0)\leq L_F(0)\leq K\ell_F(0).
\]

\begin{theorem}\label{thm:center-distortion}
Let \(H\) be a real Hilbert space with \(\dim H\geq2\), and fix an orthonormal pair \(e_1,e_2\in H\). For \(0<\tau\leq1\), define
\[
 D_\tau(t)=\sqrt{\cos^2 t+\tau^2\sin^2 t}
\]
and
\begin{equation}\label{eq:center-qc-boundary}
 \Phi_\tau(e^{it})=
 \frac{\cos t\,e_1+\tau\sin t\,e_2}{D_\tau(t)}.
\end{equation}
Let \(F_\tau=P[\Phi_\tau]\). Then \(F_\tau:\D\to\B_H\) is harmonic, \(F_\tau(0)=0\), and
\[
 F_{\tau,x}(0)=\alpha(\tau)e_1,
 \qquad
 F_{\tau,y}(0)=\beta(\tau)e_2,
\]
where
\begin{equation}\label{eq:alpha-beta-tau}
 \alpha(\tau)=\frac1\pi\int_0^{2\pi}
 \frac{\cos^2 t}{D_\tau(t)}\,dt,
 \qquad
 \beta(\tau)=\frac1\pi\int_0^{2\pi}
 \frac{\tau\sin^2 t}{D_\tau(t)}\,dt.
\end{equation}
Both quantities are positive, and \(\alpha(\tau)\geq\beta(\tau)\).

For \(K\geq1\), let \(M_K\) denote the supremum of \(L_F(0)\) over all harmonic maps \(F:\D\to\B_H\) with \(F(0)=0\) that are \(K\)-quasiconformal at the origin in the pointwise sense.
Then
\begin{equation}\label{eq:center-qc-sharp-bound}
 M_K=\inf_{0<\tau\leq1}\frac{2J(\tau)}{1+\tau/K},
 \qquad
 J(\tau)=\int_\T D_\tau(t)\dm.
\end{equation}
For \(K=1\), the unique minimizing parameter is \(\tau_1=1\). For every \(K>1\), there is a unique minimizer \(\tau_K\in(0,1)\), characterized by
\begin{equation}\label{eq:center-qc-param}
 \frac{\alpha(\tau_K)}{\beta(\tau_K)}=K.
\end{equation}
Moreover,
\[
 M_K=\alpha(\tau_K),
\]
and \(F_{\tau_K}\) attains the bound. Equality holds if and only if the
map has the form \(\mathcal U\circ F_{\tau_K}\circ\mathcal R\), where \(\mathcal U:H\to H\) is an
orthogonal transformation and \(\mathcal R\) is a rotation or reflection of
\(\D\).
\end{theorem}

\begin{proof}
\noindent\emph{Step 1: the maps \(F_\tau\) and their principal stretchings.}
The boundary values in \eqref{eq:center-qc-boundary} have norm one. Since
\[
 D_\tau(t+\pi)=D_\tau(t),
 \qquad
 \Phi_\tau(e^{i(t+\pi)})=-\Phi_\tau(e^{it}),
\]
their mean is zero. Hence \(F_\tau(0)=0\), and the Poisson integral takes values in the closed unit ball. In fact, \(F_\tau(\D)\subset\B_H\). If \(|F_\tau(z_0)|=1\) at an interior point, then the scalar harmonic function
\[
 z\longmapsto \langle F_\tau(z),F_\tau(z_0)\rangle\leq1
\]
attains its maximum at \(z_0\) and is therefore constant, contradicting
\(F_\tau(0)=0\). The Poisson formulas for the differential at the origin
give \eqref{eq:alpha-beta-tau}; the mixed terms vanish because
$\cos t\sin t/D_\tau(t)$ is odd.

To prove \(\alpha(\tau)\geq\beta(\tau)\), we reduce both integrals to
\([0,\pi/2]\) and replace \(t\) by \(\pi/2-t\) in the integral defining
\(\beta\). For \(0<t<\pi/2\), we cancel the common positive factor
\(\cos^2t\). The required pointwise comparison is
\[
 \frac{1}{\sqrt{\cos^2 t+\tau^2\sin^2 t}}
 \geq
 \frac{\tau}{\sqrt{\sin^2 t+\tau^2\cos^2 t}},
\]
whose square is equivalent to \(\sin^2t(1-\tau^4)\geq0\).

\medskip
\noindent\emph{Step 2: a dual bound for an arbitrary admissible map.}
We take an arbitrary map \(F\) in the class defining \(M_K\). By Lemma~\ref{lem:hilbert-boundary}, we write \(F=P[\Phi]\) with \(\Phi\in L^\infty(\T,H)\) and \(|\Phi|\leq1\) almost everywhere. After precomposing with a rotation or reflection of \(\D\) and applying an orthogonal transformation of \(H\), we may choose the principal stretching directions so that
\[
 F_x(0)=L_F(0)e_1,\qquad F_y(0)=\ell_F(0)e_2.
\]
For every \(0<\tau\leq1\), the pointwise Cauchy-Schwarz inequality gives
\[
 \langle\Phi(e^{it}),\cos t\,e_1+\tau\sin t\,e_2\rangle
 \leq D_\tau(t).
\]
Integration and the Poisson formulas for the differential at the origin yield
\begin{equation}\label{eq:center-dual-bound}
 L_F(0)+\tau\ell_F(0)\leq2J(\tau).
\end{equation}
Since \(L_F(0)\leq K\ell_F(0)\), we have \(\ell_F(0)\geq L_F(0)/K\). Therefore
\begin{equation}\label{eq:center-alpha-dual}
 L_F(0)\leq\frac{2J(\tau)}{1+\tau/K},
 \qquad 0<\tau\leq1.
\end{equation}
Taking the infimum in \eqref{eq:center-alpha-dual} and then the supremum
over the admissible maps yields
\[
 M_K\leq\inf_{0<\tau\leq1}
 \frac{2J(\tau)}{1+\tau/K}.
\]

\medskip
\noindent\emph{Step 3: existence and uniqueness of the minimizing parameter
for \(K>1\).}
The function in \eqref{eq:center-qc-sharp-bound} extends continuously to
\([0,1]\) by
\[
 J(0)=\int_\T|\cos t|\dm=\frac2\pi.
\]
For \(\tau>0\), differentiation under the integral sign gives
\[
 J'(\tau)=\frac{\beta(\tau)}2,
 \qquad
 J''(\tau)=\int_\T
 \frac{\cos^2t\,\sin^2t}{D_\tau(t)^3}\dm>0.
\]
To compute the right derivative at the origin, we observe that
\[
 0\leq \frac{D_\tau(t)-|\cos t|}{\tau}
 =\frac{\tau\sin^2t}{D_\tau(t)+|\cos t|}
 \leq |\sin t|.
\]
For almost every \(t\), the quotient tends to \(0\) as \(\tau\to0^+\). Dominated convergence therefore gives \(J'(0+)=0\), so the right derivative at \(0\) of \(2J(\tau)/(1+\tau/K)\) is \(-4/(\pi K)\). Also \(J(1)=1\) and \(J'(1)=1/2\), so the left derivative at \(1\) is \(K(K-1)/(K+1)^2\). Thus, for \(K>1\), every minimizer lies in \((0,1)\).

The identity
\[
 2J(\tau)=\alpha(\tau)+\tau\beta(\tau)
\]
gives
\[
 \frac{\alpha(\tau)}{\beta(\tau)}
 =\frac{J(\tau)}{J'(\tau)}-\tau.
\]
Hence
\[
 \frac{d}{d\tau}\left(\frac{\alpha(\tau)}{\beta(\tau)}\right)
 =-\frac{J(\tau)J''(\tau)}{J'(\tau)^2}<0,
 \qquad 0<\tau<1.
\]
Moreover, \(\alpha(\tau)\to4/\pi\) and \(\beta(\tau)\to0\) as
\(\tau\to0^+\). This follows by dominated convergence from
$\cos^2t/D_\tau(t)\leq|\cos t|$ and
$\tau\sin^2t/D_\tau(t)\leq|\sin t|$. Since
\(\alpha(1)=\beta(1)=1\), the ratio \(\alpha(\tau)/\beta(\tau)\)
decreases continuously from \(+\infty\) to \(1\), and
\eqref{eq:center-qc-param} has a unique solution \(\tau_K\in(0,1)\)
for every \(K>1\).

\medskip
\noindent\emph{Step 4: sharpness and equality for \(K>1\).}
We have
\[
 \frac{d}{d\tau}\left(\frac{2J(\tau)}{1+\tau/K}\right)
 =\frac{\beta(\tau)-\alpha(\tau)/K}{(1+\tau/K)^2}.
\]
The derivative changes sign from negative to positive precisely at \(\tau_K\); hence \(\tau_K\) is the unique minimizer. At this parameter,
\[
 \frac{2J(\tau_K)}{1+\tau_K/K}=\alpha(\tau_K),
\]
and the map \(F_{\tau_K}\) satisfies
\(\alpha(\tau_K)/\beta(\tau_K)=K\). It therefore attains the sharp bound.

If a map in the class defining \(M_K\) attains equality, then equality holds in \eqref{eq:center-dual-bound} at \(\tau_K\) and \(\ell_F(0)=L_F(0)/K\). We put
\[
 q_K(t)=\cos t\,e_1+\tau_K\sin t\,e_2.
\]
Equality in \eqref{eq:center-dual-bound} gives
\[
 \int_\T
 \bigl(D_{\tau_K}(t)-\langle\Phi(e^{it}),q_K(t)\rangle\bigr)\dm=0.
\]
The integrand is nonnegative, so it vanishes almost everywhere. Since
$D_{\tau_K}(t)=|q_K(t)|>0$ and $|\Phi|\leq1$, we have
\[
 \left|\Phi(e^{it})-\frac{q_K(t)}{D_{\tau_K}(t)}\right|^2
 \leq
 2\left(
 1-\left\langle
 \Phi(e^{it}),\frac{q_K(t)}{D_{\tau_K}(t)}
 \right\rangle
 \right)
 =0
\]
almost everywhere. Hence
\[
 \Phi(e^{it})=
 \frac{\cos t\,e_1+\tau_K\sin t\,e_2}{D_{\tau_K}(t)}
\]
almost everywhere in the normalized coordinates. Poisson uniqueness gives
$F=F_{\tau_K}$ there; undoing the coordinate changes gives the stated
equality cases. Conversely, orthogonal transformations of $H$ and
rotations or reflections of $\D$ preserve the target ball, the value at
the origin, and both principal stretchings. Every map in the stated
family therefore attains equality.

\medskip
\noindent\emph{Step 5: the conformal endpoint \(K=1\).}
For \(K=1\), the global minimum is obtained directly from
\[
 D_\tau(t)\geq\cos^2t+\tau\sin^2t,
 \qquad 0\leq\tau\leq1.
\]
After squaring, this is equivalent to
\[
 \cos^2t\,\sin^2t\,(1-\tau)^2\geq0.
\]
Consequently,
\[
 J(\tau)\geq\int_\T
 \bigl(\cos^2t+\tau\sin^2t\bigr)\dm
 =\frac{1+\tau}{2},
\]
and \(2J(\tau)/(1+\tau)\geq1\), with equality only at \(\tau=1\). Thus \(M_1=1\), attained by
\[
 \Phi_1(e^{it})=\cos t\,e_1+\sin t\,e_2.
\]
The equality argument above applies verbatim with \(\tau_1=1\), which
completes the proof for the conformal endpoint.
\end{proof}

In the planar case \(H=\R^2\), identify \(\operatorname{span}\{e_1,e_2\}\)
with \(\C\). The boundary map
\[
 \phi_\tau(e^{it})=
 \frac{\cos t+i\tau\sin t}{D_\tau(t)}
\]
has modulus one, and
\[
 \frac{d}{dt}\arg\phi_\tau(e^{it})
 =\frac{\tau}{D_\tau(t)^2}>0.
\]
Its argument increases by \(2\pi\) as \(t\) runs from \(0\) to \(2\pi\);
hence \(\phi_\tau\) is an orientation-preserving homeomorphism of \(\T\).
The Rad\'o-Kneser-Choquet theorem therefore shows that
\(F_\tau=P[\phi_\tau]\) is an orientation-preserving harmonic
diffeomorphism of \(\D\) onto itself; see, for example, \cite{DurenBook}. Hence the optimal constant \(M_K\) is unchanged if, in the planar case, the class defining \(M_K\) is restricted to orientation-preserving harmonic diffeomorphisms.

In the planar orientation-preserving case, the distortion result can be
expressed directly in terms of the ratio of the two complex derivatives.

\begin{corollary}\label{cor:center-beltrami}
Let \(0\leq\varepsilon<1\) and set
\[
 K_\varepsilon=\frac{1+\varepsilon}{1-\varepsilon}.
\]
Among all orientation-preserving harmonic maps \(f:\D\to\D\) satisfying \(f(0)=0\) and
\[
 \frac{|f_{\overline z}(0)|}{|f_z(0)|}=\varepsilon,
\]
one has
\begin{equation}\label{eq:center-beltrami-operator}
 \sup \|df_0\|=M_{K_\varepsilon},
\end{equation}
and
\begin{equation}\label{eq:center-beltrami-gradient}
 \sup \frac{|\nabla f(0)|}{\sqrt2}
 =\frac{\sqrt{1+\varepsilon^2}}{1+\varepsilon}\,M_{K_\varepsilon}.
\end{equation}
Both suprema are attained by the planar map
\(F_{\tau_{K_\varepsilon}}\).
\end{corollary}

\begin{proof}
For an orientation-preserving planar differential,
\[
 L_f(0)=|f_z(0)|+|f_{\overline z}(0)|,
 \qquad
 \ell_f(0)=|f_z(0)|-|f_{\overline z}(0)|.
\]
Hence the displayed condition is equivalent to
\[
 \frac{L_f(0)}{\ell_f(0)}=K_\varepsilon.
\]
Theorem~\ref{thm:center-distortion} gives \(L_f(0)\leq M_{K_\varepsilon}\), while its planar extremal mapping satisfies \(\alpha(\tau_{K_\varepsilon})/\beta(\tau_{K_\varepsilon})=K_\varepsilon\), and therefore has precisely the prescribed value of \(\varepsilon\). This proves \eqref{eq:center-beltrami-operator}. Finally,
\[
 \frac{|\nabla f(0)|}{\sqrt2}
 =\sqrt{|f_z(0)|^2+|f_{\overline z}(0)|^2}
 =\frac{\sqrt{1+\varepsilon^2}}{1+\varepsilon}\,L_f(0),
\]
which gives \eqref{eq:center-beltrami-gradient} and the equality statement.
\end{proof}

The formula in Theorem~\ref{thm:center-distortion} can be written in classical special functions. For \(0\leq k<1\), let
\[
 \mathbf K(k)=\int_0^{\pi/2}\frac{dt}{\sqrt{1-k^2\sin^2t}},
 \qquad
 \mathbf E(k)=\int_0^{\pi/2}\sqrt{1-k^2\sin^2t}\,dt
\]
denote the complete elliptic integrals of the first and second kind.

\begin{corollary}\label{cor:center-elliptic}
Let \(0<\tau<1\) and \(k=\sqrt{1-\tau^2}\). Then
\begin{equation}\label{eq:center-elliptic-J}
 J(\tau)=\frac2\pi\mathbf E(k),
\end{equation}
and
\begin{equation}\label{eq:center-elliptic-alpha-beta}
 \alpha(\tau)=\frac4\pi
 \frac{\mathbf E(k)-\tau^2\mathbf K(k)}{1-\tau^2},
 \qquad
 \beta(\tau)=\frac{4\tau}{\pi}
 \frac{\mathbf K(k)-\mathbf E(k)}{1-\tau^2}.
\end{equation}
Consequently, for \(K>1\), the parameter \(\tau_K\) is the unique solution of
\begin{equation}\label{eq:center-elliptic-parameter}
 \frac{\mathbf E(k)-\tau^2\mathbf K(k)}
 {\tau\bigl(\mathbf K(k)-\mathbf E(k)\bigr)}=K,
 \qquad k=\sqrt{1-\tau^2},
\end{equation}
and \(M_K\) is obtained by substituting \(\tau=\tau_K\) in the first formula of \eqref{eq:center-elliptic-alpha-beta}.
\end{corollary}

\begin{proof}
By symmetry,
\[
 J(\tau)=\frac2\pi\int_0^{\pi/2}
 \sqrt{1-k^2\sin^2t}\,dt,
\]
which proves \eqref{eq:center-elliptic-J}. Also,
\[
 \mathbf K(k)-\mathbf E(k)
 =k^2\int_0^{\pi/2}
 \frac{\sin^2t}{\sqrt{1-k^2\sin^2t}}\,dt.
\]
This gives the formula for \(\beta\). Subtracting the last integral from \(\mathbf K(k)\) gives
\[
 \int_0^{\pi/2}
 \frac{\cos^2t}{\sqrt{1-k^2\sin^2t}}\,dt
 =\frac{\mathbf E(k)-\tau^2\mathbf K(k)}{1-\tau^2},
\]
and hence the formula for \(\alpha\). Equation \eqref{eq:center-elliptic-parameter} is \eqref{eq:center-qc-param} in these variables.
\end{proof}

The elliptic formulas also permit a direct comparison with classical
ellipse coefficients.

\begin{remark}\label{rem:classical-ellipse-coefficients}
For \(0<\tau<1\), identify
\(\operatorname{span}\{e_1,e_2\}\) with \(\C\), and write
\[
 \phi_\tau(e^{it})=
 \frac{\cos t+i\tau\sin t}{\sqrt{\cos^2t+\tau^2\sin^2t}}.
\]
With the Fourier convention used here,
\[
 \widehat{\phi_\tau}(1)=\frac{\alpha(\tau)+\beta(\tau)}2,
 \qquad
 \widehat{\phi_\tau}(-1)=\frac{\alpha(\tau)-\beta(\tau)}2.
\]
Let \(k=\sqrt{1-\tau^2}\). Substitution of \eqref{eq:center-elliptic-alpha-beta} gives
\[
 \widehat{\phi_\tau}(1)=
 \frac{2}{\pi}\frac{\mathbf E(k)+\tau\mathbf K(k)}{1+\tau},
 \qquad
 \widehat{\phi_\tau}(-1)=
 \frac{2}{\pi}\frac{\mathbf E(k)-\tau\mathbf K(k)}{1-\tau}.
\]
These expressions coincide with Wegmann's formula~(55), on p.~178 of
\cite{Wegmann1993}, for the pair of Fourier indices $1$ and $-1$.
The parameter $p_{\mathrm W}$ in that formula is the ratio of the minor
to the major semiaxis; setting $p_{\mathrm W}=\tau\in(0,1)$ gives
\[
 A_1=\widehat{\phi_\tau}(1),
 \qquad
 A_{-1}=\widehat{\phi_\tau}(-1).
\]
The ellipse $\Gamma_\tau(t)=\sin t-i\tau\cos t$ has semiaxes
\(1\) and \(\tau\), and
\[
 \frac{\Gamma_\tau'(t)}{|\Gamma_\tau'(t)|}
 =\frac{\cos t+i\tau\sin t}
 {\sqrt{\cos^2t+\tau^2\sin^2t}}
 =\phi_\tau(e^{it}).
\]
Hence, up to a shift of the boundary parameter and an orthogonal
transformation of the target plane, \(\phi_\tau\) is the unit
tangent-direction map of the corresponding ellipse. Thus this planar boundary
family and the two coefficients displayed above already occur in the
classical orientation-preserving setting.

In Theorem~\ref{thm:center-distortion}, this family solves the constrained
problem; its equality statement applies to
the full class of harmonic maps into $\B_H$.
\end{remark}

The formula also yields the monotonicity and endpoint behavior of \(M_K\).

\begin{corollary}\label{cor:center-monotonicity}
The function \(K\mapsto\tau_K\) is strictly decreasing on \([1,\infty)\), while \(K\mapsto M_K\) is strictly increasing. Moreover,
\[
 M_1=1,
 \qquad
 1<M_K<\min\left\{\frac{2K}{K+1},\frac4\pi\right\}\quad(K>1),
 \qquad
 \lim_{K\to\infty}M_K=\frac4\pi.
\]
\end{corollary}

\begin{proof}
The ratio \(\alpha(\tau)/\beta(\tau)\) is strictly decreasing, so
\(K\mapsto\tau_K\) is strictly decreasing. Since
\[
 2J(\tau)=\alpha(\tau)+\tau\beta(\tau)
 \qquad\text{and}\qquad
 \beta(\tau)=2J'(\tau),
\]
we have
\[
 \alpha(\tau)=2J(\tau)-2\tau J'(\tau),
\]
and therefore
\[
 \alpha'(\tau)=-2\tau J''(\tau)<0,
 \qquad 0<\tau<1.
\]
Since \(M_K=\alpha(\tau_K)\), the function \(K\mapsto M_K\) is strictly
increasing.

For $K>1$, the unique minimizing parameter satisfies $\tau_K<1$.
The value of the minimized function at $\tau=1$ is $2K/(K+1)$.
Uniqueness therefore gives
\[
 M_K<\frac{2K}{K+1}<K.
\]

The upper bound \(M_K\leq4/\pi\) follows from \eqref{eq:center-qc-sharp-bound} by letting \(\tau\to0^+\). Conversely, we fix \(\tau>0\). For every sufficiently large \(K\), the map
\(F_\tau\) satisfies the pointwise \(K\)-distortion condition, and therefore
\(M_K\geq\alpha(\tau)\). Hence
\[
 \liminf_{K\to\infty}M_K\geq\alpha(\tau).
\]
Passing to the limit \(\tau\to0^+\) gives the reverse inequality because \(\alpha(\tau)\to4/\pi\).
Therefore \(M_K\to4/\pi\) as \(K\to\infty\). Since \(M_K\) is strictly
increasing, it cannot attain this limiting value at a finite \(K\). Together
with \(M_1=1\), this proves
\[
 1<M_K<\min\left\{\frac{2K}{K+1},\frac4\pi\right\},
 \qquad K>1.
\]
\end{proof}

The limiting constant has a direct geometric interpretation. As
\(K\to\infty\), one has \(\tau_K\to0\) and, for almost every \(t\),
\[
 \Phi_{\tau_K}(e^{it})
 \longrightarrow \operatorname{sgn}(\cos t)e_1.
\]
Thus the normalized elliptic boundary function converges almost everywhere
to the two values \(e_1\) and \(-e_1\), while
\[
 \alpha(\tau_K)\longrightarrow\frac4\pi,
 \qquad
 \beta(\tau_K)\longrightarrow0.
\]
Dominated convergence also gives $L^2(\T,H)$ convergence of the boundary
functions, and their Poisson integrals converge locally uniformly in $\D$.
At \(K=1\), one has \(\tau_1=1\) and
\(\Phi_1(e^{it})=\cos t\,e_1+\sin t\,e_2\). The extremal family therefore
connects the conformal disk case continuously to the one-dimensional harmonic
map that attains equality in Proposition~\ref{prop:hilbert-unrestricted-center}.

At the conformal endpoint, the parameter and the optimal constant have the
following asymptotic expansions.

\begin{corollary}\label{cor:center-near-conformal}
As \(K\to1^+\),
\[
 \tau_K=1-2(K-1)+O\bigl((K-1)^2\bigr),
\]
and
\[
 M_K=1+\frac12(K-1)+O\bigl((K-1)^2\bigr).
\]
\end{corollary}

\begin{proof}
We set \(Q(\tau)=\alpha(\tau)/\beta(\tau)\). Since
\[
 J(1)=1,\qquad J'(1)=\frac12,\qquad
 J''(1)=\int_\T\cos^2t\sin^2t\dm=\frac18,
\]
we have
\[
 Q'(1)=-\frac{J(1)J''(1)}{J'(1)^2}=-\frac12.
\]
The integral formulas for $\alpha$ and $\beta$ extend smoothly to all
$\tau>0$. Since \(Q\) is smooth in a neighborhood of \(\tau=1\) and
\(Q'(1)=-1/2\neq0\), the implicit function theorem applied to
\(Q(\tau_K)=K\) gives the first expansion. Also,
\(\alpha'(1)=-2J''(1)=-1/4\), and \(M_K=\alpha(\tau_K)\), which gives
the second.
\end{proof}

\section{Concluding remarks}\label{sec:comparison}

The target-independent M\"obius-weighted Hilbert identity turns the
prescribed-value derivative extremal problem into a boundary-pairing
problem. For a round affine section,
Corollary~\ref{cor:mobius-extremality} gives a necessary and sufficient
integral criterion for the M\"obius parametrization to be extremal, while
Proposition~\ref{prop:section-support-criterion} gives a geometric
sufficient condition through supporting hyperplanes of the ambient target.

The examples in Section~\ref{sec:target-geometry} distinguish three
properties that need not coincide: roundness of an affine section,
extremality of its M\"obius parametrization, and the supporting-hyperplane
condition. Roundness alone does not imply extremality; the supporting
condition may hold in a non-ball target; and
Theorem~\ref{thm:balanced-complex-lines} shows that extremality on central
complex lines can persist even when the supporting condition fails.

For Hilbert balls, the same pairing yields both the sharp
prescribed-value conformal estimate and the pointwise distortion estimates.
Theorem~\ref{thm:ball-boundary-stability} gives an exact deficit identity
and quantitative boundary stability. At the center,
Theorem~\ref{thm:center-distortion} determines the optimal value under a
pointwise distortion constraint, the unique optimizing parameter, and all
equality cases. The corresponding prescribed-value optimization problem
for a general nonconformal differential away from the center is not
pursued here.

\section*{Data availability}
No data were used for the research described in this article.

\section*{Declaration of competing interest}
The authors declare that they have no known competing financial interests or
personal relationships that could have appeared to influence the work reported
in this paper.

\section*{Declaration of generative AI and AI-assisted technologies in the manuscript preparation process}
During the preparation of this work, the authors used Paperpal (integrated within the Overleaf platform) to assist with language editing, academic tone improvement, and consistency checking. After using this tool, the authors reviewed and edited the content as needed and take full
responsibility for the content of the published article.

\bibliographystyle{plain}
\bibliography{main}

\end{document}